\documentclass{amsart}

\usepackage{graphicx}
\usepackage{amssymb}
\usepackage{amsmath}
\usepackage{amsfonts}
\usepackage{array}
\usepackage{multirow}
\usepackage{makecell}
\usepackage{hhline}
\usepackage{multirow,bigdelim}
\usepackage{mathtools}
\usepackage{mdframed}
\usepackage{lipsum}
\usepackage{enumitem}
\usepackage{stackrel}
\usepackage{tabu}
\usepackage{tikz}
\usepackage{algpseudocode}
\usepackage{algorithm}
\usetikzlibrary{arrows}

\newtheorem{theorem}{Theorem}[section]
\newmdtheoremenv{boxtheo}[theorem]{Theorem}
\newtheorem{proposition}[theorem]{Proposition}
\newtheorem{corollary}{Corollary}[theorem]
\newtheorem{lemma}[theorem]{Lemma}
\theoremstyle{definition}
\newtheorem{definition}[theorem]{Definition}
\newtheorem{example}[theorem]{Example}

\theoremstyle{remark}
\newtheorem{remark}{Remark}

\newcommand{\pp}{$^\prime$}

\newcommand{\realify}[1]{[#1]_\mathbb{R}}

\newcommand{\mi}{\!\!\!\scalebox{0.75}[1.0]{$-$}}
\newcommand{\mis}{\,\hspace{-0.12cm}\scalebox{0.4}[0.7]{\( - \)}}
\newcommand{\RN}[1]{%
  \textup{\uppercase\expandafter{\romannumeral#1}}%
}
\newcommand{\rn}[1]{%
  \textup{\lowercase\expandafter{\romannumeral#1}}%
}
\newcommand*\circled[1]{\tikz[baseline=(char.base)]{
            \node[shape=circle,draw,inner sep=2pt] (char) {#1};}}

\DeclareMathOperator{\cent}{cent}
\DeclareMathOperator{\cosq}{cosq}
\DeclareMathOperator{\sol}{sol}
\DeclareMathOperator{\cosol}{cosol}

\begin{document}

%\title{Tangent spaces of the isometry groups}
\title[Solutions to the matrix equation $G^*JG=J$]{On the structure of the solutions to the matrix equation $G^*JG=J$}

\author{Alan Edelman}
\address{Department of Mathematics and Computer Science \& AI Laboratory, Massachusetts Institute of Technology, Cambridge, Massachusetts, 02139}
\email{edelman@mit.edu}

\author{Sungwoo Jeong}
\address{Department of Mathematics, Massachusetts Institute of Technology, Cambridge, Massachusetts, 02139}
\email{sw2030@mit.edu}

%\date{December 1, 2021.}

\begin{abstract}
We study the mathematical structure of the solution set (and its tangent space) to the matrix equation $G^*\!JG=J$ for a given square matrix $J$. In the language of pure mathematics, this is a Lie group which is the isometry group for a bilinear (or a sesquilinear) form. Generally these groups are described as intersections of a few special groups.

\par The tangent space to $\{G: G^*\!JG=J \}$ consists of solutions to the linear matrix equation $X^*\!J+JX=0$. For the complex case, the solution set of this linear equation was computed by De Ter{\'a}n and Dopico.

\par We found that on its own, the equation $X^*\!J+JX=0$ is hard to solve. By throwing into the mix the complementary linear equation $X^*\!J-JX=0$, we find that the direct sum of the two solution sets is an easier to compute linear space. Thus, we obtain the two solution sets from projection maps. Not only is it possible to now solve the original problem, but we can approach the broader algebraic and geometric structure. One implication is that the two equations form an $\mathfrak{h}$ and $\mathfrak{m}$ pair familiar in the study of pseudo-Riemannian symmetric spaces. 
%We show that the two equations reveal an interesting geometric property. 

\par We explicitly demonstrate the computation of the solutions to the equation $X^*\!J\pm XJ=0$ for real and complex matrices. However, real, complex or quaternionic case with an arbitrary involution (e.g., transpose, conjugate transpose, and the various quaternion transposes) can be effectively solved with the same strategy. We provide numerical examples and visualizations.
\end{abstract}

%Coupled with the solution of a similar complementary equation, the two solution %spaces 
%Algebraically, 
%For a given bilinear or sesquilinear form, the isometry group preserving the form %is a matrix Lie group. %In this work we propose a structural approach to a broader problem, including the %aforementioned complex case. 
%We also provide a sketch of a numerical algorithm for solving the equation.
%[need to discuss more]
%(in the paper, mention alternatives like guptri, but maybe don't implement now)

%%%%%MUST TO DO BEFORE IT GOES OUT:
%%%%% ADJUST ALL \! in G^*JG and X^*J
%%%%% theorems->Theorems

\maketitle

\section{Introduction}
We study the structure of the matrix group $\{G : G$ invertible\footnote{In the following we will always assume invertibility of $G$.}$\text{, } G^*\!JG =J\}$ and its tangent space at the identity $\{X:X^*\!J+JX =0\}$. We assume $J$ is a given square matrix and the ``star" superscript, $G^*$, is either $G^T$ (the usual matrix transposition) or $G^H$ (conjugate transposition). Previous work related to this question may be found in \cite{de2011conjtrans,de2011trans,dmytryshyn2015change,mackey2003structured,mackey2005structured,riehm1974equivalence,szechtman2005structure}.

%We are motivated to describe these Lie groups and tangent spaces with details and possibly with numerical application. 

\par The group $\{G : G^*\!JG = J\}$ is often called the \textit{automorphism group} or the \textit{isometry group} (of a bilinear/sesquilinear form) \cite{jacobson2009basic,lang2002algebra,mackey2003structured}. Given a bilinear form $\langle x, y \rangle_J = x^TJy$ or a sesquilinear form $\langle x, y\rangle_J = x^H Jy$, the automorphism group is the collection of linear operators that preserve this form, i.e., $\langle x, y\rangle_J = \langle Gx, Gy \rangle_J$. Representing the linear operators as matrices, they are the solutions $G$ to the matrix quadratic equation $G^*\!JG = J$. For some special $J$'s, the automorphism groups are well known as classical Lie groups \cite{weyl1946classical}. 

Three closely related questions are the strict equivalence of pencils of the form $J+\lambda J^T$ \cite{dmytryshyn2015change,dmytryshyn2013codimension,schroder2006canonical}, the orbits of matrices under the congruence transformation $J \mapsto KJK^*$ where $K$ is invertible \cite{horn2007canonical,turnbull1932introduction}, and the question of similar automorphism groups (the focus of this paper).
% Maybe add more references?

The familiar example from elementary linear algebra would be $J\!=\!I_n$, where the automorphism group (of the real bilinear form) is the orthogonal group $\text{O}(n)$  (all orthogonal matrices). Another key example is $J = \begin{bsmallmatrix}\,\,0 & I_n \\ -I_n & 0\end{bsmallmatrix}$, namely the skew-symmetric bilinear form, where the automorphism group is the symplectic group $\text{Sp}(2n, \mathbb{R})$ that appears in symplectic geometry and classical mechanics. Let us define notations which come up frequently for real and complex cases.
\begin{definition}
For a complex $J\in\mathbb{C}^{n\times n}$, define groups $G_J$ and $G_J^H$ as follows:
\begin{equation*}
    G_J := \{G:G^T\!JG = J, G\in\mathbb{C}^{n\times n}\}, \hspace{0.5cm} G_J^H := \{G:G^H\!JG = J, G\in\mathbb{C}^{n\times n}\}.
\end{equation*}
Also for real $J\in\mathbb{R}^{n\times n}$, define $G_J^\mathbb{R} := \{G: G^T\!JG=J, G\in\mathbb{R}^{n\times n}\}$.
\end{definition}

\begin{figure}[h]
    \centering
    \title{\large Visualizations for $G_J^\mathbb{R}$, generic $J\in\mathbb{R}^{4\times 4}$}
    \fbox{\includegraphics[width=4.9in]{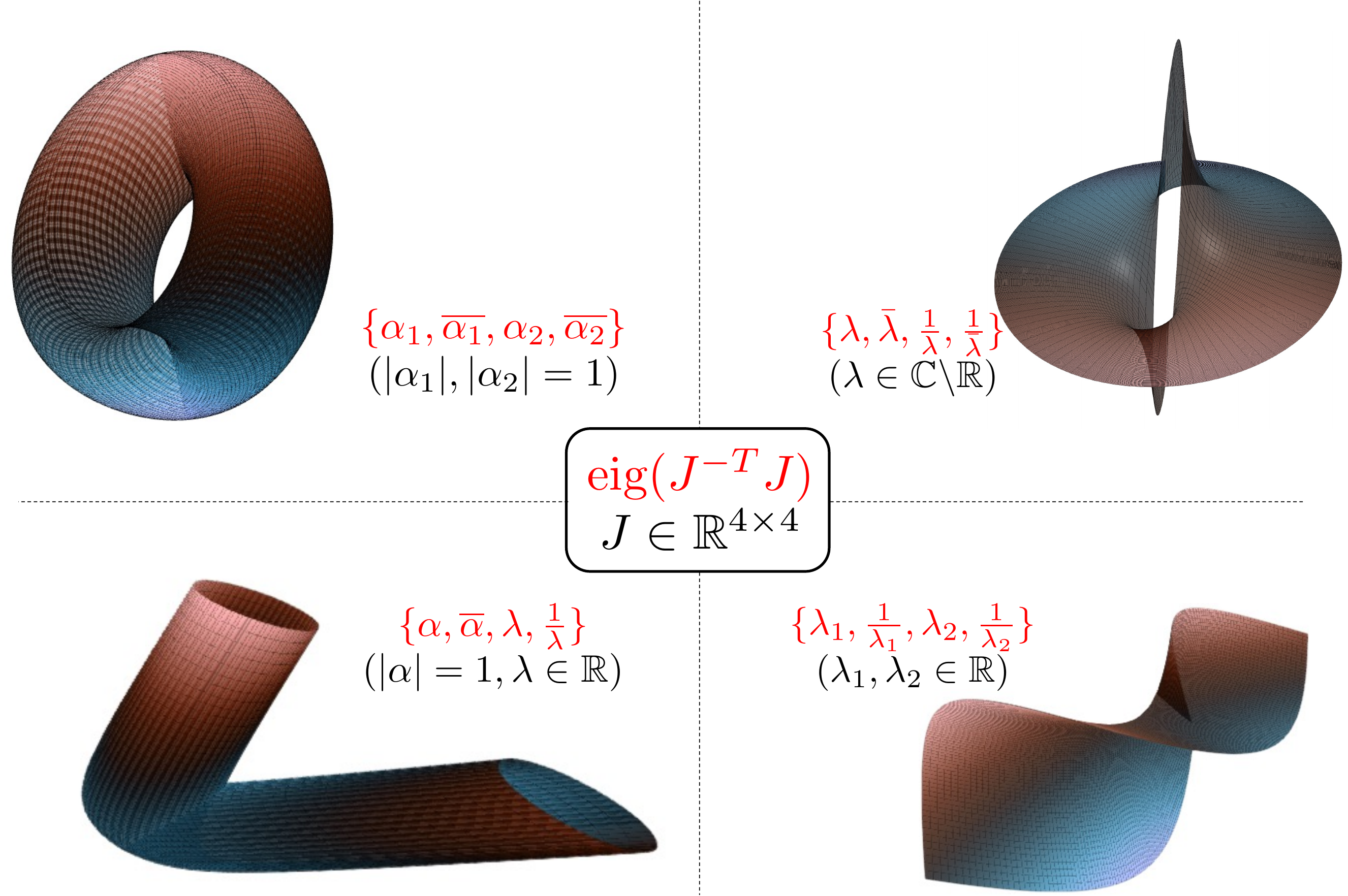}}
    \caption{Four types of generic $4\times 4$ real $J$'s classified by the eigenvalue characteristic of $J^{-T}\!J$ and the visualizations of corresponding $G_J^\mathbb{R}$. Plots are three dimensional projections (from 16 dimensions) of the identity component of the surface $G_J^\mathbb{R}$. See Section \ref{sec:4x4generic} for more details.}
    \label{fig:4x4plots} 
\end{figure}

\begin{remark}
For a $4\times 4$ real $J$, the case $J\!=\!I_4$ which leads to the orthogonal group $\text{O}(4)$ is not at all a generic case.\footnote{Note that $J$ is generic if it is nonsingular and $J^{-T}\!J$ has no double eigenvalues, which is essentially excluding the ``special" cases we will describe in Sections \ref{sec:background} to \ref{sec:realdimcount}.} The reader would do well to wonder what the generic cases look like. There are four possible generic cases whose solution in each case is a two dimensional surface in 16 dimensions as illustrated in Figure \ref{fig:4x4plots}, with details following in Section \ref{sec:4x4generic}. The following Example \ref{ex:4x4example} considers one of these generic cases.
\end{remark}

\par For an arbitrary $J$ the group $\{G:G^*\!JG=J\}$ is known to be the intersection of the groups determined by the symmetric and skew-symmetric (Hermitian and skew-Hermitian if $*=H$) parts of $J$ \cite[p.92]{rossmann2002lie}. For example if $J$ is real $2n\times 2n$, decompose $J$ into its symmetric/skew-symmetric parts $J=S+A$ and assume $S, A$ nonsingular. Let $(p,q)$ be the signature of $S$. We obtain the two groups $G_S^\mathbb{R}$ and $G_A^\mathbb{R}$ each isomorphic (not necessarily identical) to $\text{O}(p, q)$ and $\text{Sp}(2n, \mathbb{R})$, such that the intersection equals $G_J^\mathbb{R}$. Correspondingly, the tangent space $\{X:X^T\!J+JX=0\}$ is the intersection of the tangent spaces of the two groups. 

\par Though it is known that the groups are intersections of two other groups, it seems if one wants to find the intersection computationally one might have to perform linear algebra operations such as the SVD on $n^2\times n^2$ matrices with a prohibitive dense complexity of $O(n^6)$. In this paper, we demonstrate an approach using the generalized eigenstructure that directly provides a basis for the intersection.

% 4x4 example with positive definite symmetric part, Showing the basis of O(4) and Sp(4, R) explicitly is the goal of this example. 
\begin{example}\label{ex:4x4example}
In this example we define a $J\in\mathbb{R}^{4\times 4}$ where the symmetric part of $J$ is positive definite. The group $G_J^\mathbb{R}$ is the intersection of groups similar to $\text{O}(4)$ and $\text{Sp}(4, \mathbb{R})$. Let % $J$ and its skew-symmetric part $A$ be the following:
\begin{gather*}
    J = \text{\scalebox{0.8}[0.8]{
    $\begin{bmatrix}\begin{array}{rrrr}
    1 & 1 & \mi1 & \mi1\\ 
    \mi1 & 1 &  \cdot & \mi1\\  
    1 & \cdot & 1 & \mi1 \\ 
    1 & 1  & 1  & 1
    \end{array}\end{bmatrix}$}} = \underbrace{\text{\scalebox{0.8}[0.8]{$\begin{bmatrix}1&\cdot&\cdot&\cdot\\\cdot&1&\cdot&\cdot\\\cdot&\cdot&1&\cdot\\\cdot&\cdot&\cdot&1\end{bmatrix}$}}}_{\text{symmetric part $I_4$}} + \underbrace{\text{\scalebox{0.8}[0.8]{$\begin{bmatrix}\begin{array}{rrrr}\cdot & 1 & \mi1 & \mi1\\ \mi1 & \cdot &  \cdot & \mi1\\  1 & \cdot & \cdot & \mi1 \\ 1 & 1  & 1  & \cdot\end{array}\end{bmatrix}$}}}_{\text{skew-symmetric part $A$}} ,\\
    A = K^T\text{\scalebox{0.8}[0.8]{$\begin{bmatrix}\begin{array}{rrrr}\cdot&\cdot&1&\cdot\\\cdot&\cdot&\cdot&1\\ \mi1&\cdot&\cdot&\cdot\\\cdot&\mi1&\cdot&\cdot\end{array}\end{bmatrix}$}}K, \hspace{0.3cm}\text{ where }
    K=\text{\scalebox{0.8}[0.8]{$\begin{bmatrix}
         \cdot & 1 & 1 & 1\\2 & \cdot & \cdot & \cdot \\ 1 & 1 & 1 & \cdot \\ \cdot & 1 & \cdot & \cdot
    \end{bmatrix}$}}.
\end{gather*}
Then, the group $G_J^\mathbb{R}$ and its Lie algebra $\mathfrak{g}_J^\mathbb{R}$ are defined as,  
\begin{equation*}
    G_J^\mathbb{R} = \text{O}(4)\cap\big(K^{-1}\text{Sp}(4, \mathbb{R})K\big), \hspace{1cm} \mathfrak{g}_J^\mathbb{R} = \mathfrak{o}(4)\cap \big(K^{-1}\mathfrak{sp}(4, \mathbb{R})K\big).
\end{equation*}
$\mathfrak{o}(4) = \{\text{All skew-symmetric matrices}\}$ is 6 dimensional and $\mathfrak{sp}(4,\mathbb{R})$ is a 10 dimensional linear subspace with the following standard basis:
\begin{gather*}
    \mathfrak{sp}(4, \mathbb{R}):
    \bigg\{\begin{bsmallmatrix} 1 &  \cdot &   \cdot &  \cdot \\ \cdot &  \cdot &   \cdot &  \cdot  \\\cdot &  \cdot &  \mis 1 &  \cdot \\ \cdot &  \cdot &   \cdot &  \cdot \end{bsmallmatrix},
    \begin{bsmallmatrix} \cdot &  \cdot &  \cdot &   \cdot  \\ 1 &  \cdot &  \cdot &   \cdot  \\ \cdot &  \cdot &  \cdot &  \mis 1  \\ \cdot &  \cdot &  \cdot &   \cdot
\end{bsmallmatrix},\begin{bsmallmatrix}
 \cdot&  1 &   \cdot&  \cdot\\
 \cdot&  \cdot&   \cdot&  \cdot\\
 \cdot&  \cdot&   \cdot&  \cdot\\
 \cdot&  \cdot&  \mis 1 &  \cdot
\end{bsmallmatrix},\begin{bsmallmatrix}
 \cdot &  \cdot &  \cdot &    \cdot \\
 \cdot &  1 &  \cdot &    \cdot \\
 \cdot &  \cdot &  \cdot &    \cdot \\
 \cdot &  \cdot &  \cdot &  \mis 1 
\end{bsmallmatrix},\begin{bsmallmatrix}
 \cdot &  \cdot &  1 &  \cdot \\
 \cdot &  \cdot &  \cdot &  \cdot \\
 \cdot &  \cdot &  \cdot &  \cdot \\
 \cdot &  \cdot &  \cdot &  \cdot 
\end{bsmallmatrix},\hspace{2cm}\\
\hspace{2cm}\begin{bsmallmatrix}
 \cdot &  \cdot &  \cdot &  \cdot\\
 \cdot &  \cdot &  \cdot &  1 \\
 \cdot &  \cdot &  \cdot &  \cdot \\
 \cdot &  \cdot &  \cdot &  \cdot 
\end{bsmallmatrix},\begin{bsmallmatrix}
 \cdot &  \cdot &  \cdot &  1 \\
 \cdot &  \cdot &  1 &  \cdot \\
 \cdot &  \cdot &  \cdot &  \cdot \\
 \cdot &  \cdot &  \cdot &  \cdot 
\end{bsmallmatrix},\begin{bsmallmatrix}
 \cdot &  \cdot &  \cdot &  \cdot \\
 \cdot &  \cdot &  \cdot &  \cdot \\
 1 &  \cdot &  \cdot &  \cdot \\
 \cdot &  \cdot &  \cdot &  \cdot 
\end{bsmallmatrix},\begin{bsmallmatrix}
 \cdot &  \cdot &  \cdot &  \cdot \\
 \cdot &  \cdot &  \cdot &  \cdot \\
 \cdot &  \cdot &  \cdot &  \cdot \\
 \cdot &  1 &  \cdot &  \cdot 
\end{bsmallmatrix},\begin{bsmallmatrix}
 \cdot &  \cdot &  \cdot &  \cdot \\
 \cdot &  \cdot &  \cdot &  \cdot \\
 \cdot &  1 &  \cdot &  \cdot \\
 1 &  \cdot &  \cdot &  \cdot  \end{bsmallmatrix}\bigg\}.
\end{gather*}
Writing down a basis of $\mathfrak{o}(4)$ and $K^{-1}\mathfrak{sp}(4,\mathbb{R})K$:
\begin{gather*}
    \mathfrak{o}(4) : \bigg\{\begin{bsmallmatrix} \cdot&  1&  \cdot&  \cdot\\ \mis 1&  \cdot&  \cdot&  \cdot \\   \cdot&  \cdot&  \cdot&  \cdot \\   \cdot&  \cdot&  \cdot&  \cdot    \end{bsmallmatrix},
    \begin{bsmallmatrix} \cdot&  \cdot&  1&  \cdot \\ \cdot&  \cdot&  \cdot&  \cdot \\ \mis1&  \cdot&  \cdot&  \cdot\\  \cdot&  \cdot&  \cdot&  \cdot    \end{bsmallmatrix},
    \begin{bsmallmatrix}  \cdot&   \cdot&  \cdot&  \cdot\\ \cdot& \cdot&  1&  \cdot \\ \cdot&  \mis1&  \cdot&  \cdot \\ \cdot&  \cdot&  \cdot&  \cdot    \end{bsmallmatrix},
    \begin{bsmallmatrix} \cdot&  \cdot&  \cdot&  1 \\  \cdot&  \cdot&  \cdot&  \cdot \\  \cdot&  \cdot&  \cdot&  \cdot \\ \mis1&  \cdot&  \cdot&  \cdot  \end{bsmallmatrix},
    \begin{bsmallmatrix} \cdot&   \cdot&  \cdot&  \cdot \\ \cdot&   \cdot&  \cdot&  1 \\ \cdot&   \cdot&  \cdot&  \cdot \\ \cdot&  \mis1&  \cdot&  \cdot
    \end{bsmallmatrix},
    \begin{bsmallmatrix} \cdot&  \cdot&   \cdot&  \cdot \\ \cdot&  \cdot&   \cdot&  \cdot \\  \cdot&  \cdot&   \cdot&  1 \\ \cdot&  \cdot&  \mis1&  \cdot
    \end{bsmallmatrix}\bigg\}, \\
K^{-1}\mathfrak{sp}(4, \mathbb{R})K: \bigg\{\begin{bsmallmatrix}
  \cdot &   \cdot &   \cdot &  \cdot \\
  \cdot &   \cdot &   \cdot &  \cdot \\
 \mis1 &  \mis1 &  \mis1 &  \cdot \\
  1 &   2 &   2 &  1 
\end{bsmallmatrix}, {\scalebox{1.2}{$\frac{1}{2}$}}\begin{bsmallmatrix}
 \cdot &   1 &   1 &   1 \\
 \cdot &   \cdot &   \cdot &   \cdot \\
 \cdot &  \mis3 &  \mis1 &  \mis1 \\
 \cdot &   3 &   1 &   1 
\end{bsmallmatrix}, \begin{bsmallmatrix}
  \cdot &   \cdot &   \cdot &  \cdot \\
 \mis1 &  \mis1 &  \mis1 &  \cdot \\
  1 &   1 &   1 &  \cdot \\
  2 &   \cdot &   \cdot &  \cdot 
\end{bsmallmatrix}, \begin{bsmallmatrix}
  1 &   \cdot &  \cdot &  \cdot \\
  \cdot &  \mis1 &  \cdot &  \cdot \\
 \mis1 &   1 &  \cdot &  \cdot \\
  1 &   \cdot &  \cdot &  \cdot 
\end{bsmallmatrix}, \begin{bsmallmatrix}
 \cdot &  \cdot &  \cdot &  \cdot \\
 \cdot &  \cdot &  \cdot &  \cdot \\
 \cdot &  \cdot &  \cdot &  \cdot \\
 1 &  1 &  1 &  \cdot 
\end{bsmallmatrix},\hspace{1cm}\, \\
\hspace{3.5cm}{\scalebox{1.2}{$\frac{1}{2}$}}\begin{bsmallmatrix}
 \cdot &   1 &  \cdot &  \cdot \\
 \cdot &   \cdot &  \cdot &  \cdot \\
 \cdot &  \mis1 &  \cdot &  \cdot \\
 \cdot &   1 &  \cdot &  \cdot 
\end{bsmallmatrix}, {\scalebox{1.2}{$\frac{1}{2}$}}\begin{bsmallmatrix}
  1 &   1 &   1 &  \cdot \\
  \cdot &   \cdot &   \cdot &  \cdot \\
 \mis1 &  \mis1 &  \mis1 &  \cdot \\
  1 &   3 &   1 &  \cdot 
\end{bsmallmatrix}, \begin{bsmallmatrix}
 \cdot &   \cdot &   \cdot &   \cdot \\
 \cdot &   \cdot &   \cdot &   \cdot \\
 \cdot & 1 &   1 & 1 \\
 \cdot &  \mis1 &  \mis1 &  \mis1 
\end{bsmallmatrix}, \begin{bsmallmatrix}
  \cdot &  \cdot &  \cdot &  \cdot \\
  2 &  \cdot &  \cdot &  \cdot \\
 \mis2 &  \cdot &  \cdot &  \cdot \\
  \cdot &  \cdot &  \cdot &  \cdot
\end{bsmallmatrix}, \begin{bsmallmatrix}
  \cdot &   \cdot &   \cdot &   \cdot \\
  \cdot &   1 &   1 & 1 \\
 2 &  \mis1 &  \mis1 &  \mis1 \\
 \mis2 &   \cdot &   \cdot &   \cdot 
 \end{bsmallmatrix}\bigg\}
\end{gather*}
Computing the intersection of $\mathfrak{o}(4)$ and $K^{-1}\mathfrak{sp}(4, \mathbb{R})K$, we finally obtain a basis: 
\begin{equation*}
\mathfrak{g}_J^\mathbb{R} = \mathfrak{o}(4)\cap\,(K^{-1}\mathfrak{sp}(4, \mathbb{R})K) = \text{Span}\Bigg(
    \text{\scalebox{0.7}[0.7]{$\begin{bmatrix}\begin{array}{rrrr}
     \cdot &  1 &  \mi1 &  \mi2\\\mi1 &  \cdot &  \mi1 &  \mi1 \\ 1 &  1 &   \cdot &  \mi1 \\  2 &  1 &   1 &   \cdot 
    \end{array}\end{bmatrix}$}}, \text{\scalebox{0.7}[0.7]{$\begin{bmatrix}\begin{array}{rrrr}
    \cdot &  \mi1 &   1 &  \mi 1 \\ 1 &   \cdot &  \mi2 &   1 \\ \mi1 &   2 &   \cdot &   1 \\  1 &  \mi1 &  \mi1 &   \cdot 
    \end{array}\end{bmatrix}$}}\Bigg).
\end{equation*}
\end{example}
%%% Example code in posdefcomp.jl

\par As one can see in the above example, working with the tangent space is less complicated as it is linear. In practice, the matrices in the tangent space could be mapped to $G_J^\mathbb{R}$ by the exponential map. For details on the exponential map, see Section \ref{sec:exp}. One of our main focuses is the outline of a computation of a basis of the tangent space $\{X:X^*\!J+JX=0\}$. A key observation is that the direct sum of the two solution sets $\{X:X^*\!J\pm JX=0\}$ is a well-known and easier to compute linear spaces. In particular if $J$ is nonsingular the direct sum is the centralizer of $J^{-*}\!J$ (Corollary \ref{cor:solcosoldirectsum}). The bases for the solution sets could be computed by projection maps. Additionally we obtain a basis of $\{X:X^*\!J - JX=0\}$ as a byproduct.

\par With a more direct strategy, De Ter\'{a}n and Dopico \cite{de2011conjtrans,de2011trans} computed the solutions of $XJ+JX^T\!=\!0$ and $XJ+JX^H\!=\!0$ for complex matrices\footnote{Obviously, the transposed solution sets of \cite{de2011trans,de2011conjtrans} are equivalent to our solution sets.} using the congruence canonical form\footnote{We refer to De Ter{\'a}n's clear summary \cite{de2016canonical} for different types of congruence canonical forms.} studied by Horn and Sergeichuk \cite{horn2007canonical}. One needs only to compute the solutions for canonical $J$'s of the congruence transformation $J\mapsto K^*\!JK$ since the groups $\{G:G^*\!JG=J\}$ are similar for congruent $J$'s. They carefully worked out case-by-case solutions for each canonical form and their interactions.\footnote{The canonical forms of the congruence transformation is related to the $*$-palindromic matrix pencil $J-\lambda J^*$. The results in \cite{de2011conjtrans,de2011trans} are further extended using this idea in \cite{de2013solution}, where the authors use the Kronecker structure of general matrix pencils.} Some complicated cases were later brought to the explicit expressions in \cite{chan2013matrix,garcia2013matrix}. 

\par Our approach solves the equation $X^*\!J+JX=0$ in a more general setting (regardless of real, complex, and quaternion) by directly exploring structure. In particular, we point out that the relationship between solutions of $X^T\!J\pm JX=0$ brings to mind the structure of \textit{symmetric spaces} \cite{Helgason1978}. For a nonsingular $J$ we have 
\begin{equation*}
    \{g : g(J^{-T}J)=(J^{-T}J)g\} = \{h: h^TJ+Jh=0\} \oplus \{m : m^TJ-Jm=0\},
\end{equation*}
which is the Lie algebra decomposition of a pseudo-Riemannian symmetric space $\mathfrak{g} = \mathfrak{h}\oplus\mathfrak{m}$. Furthermore the centralizer of the \textit{cosquare} $J^{-T}\!J$ has a well-known structure we can adapt to solve the equations $X^T\!J\pm JX=0$ at hand.

%\footnote{In \cite{de2011trans,de2011conjtrans} the authors separate the singular and nonsingular part of $J$ as the congruence canonical naturally isolates the singular parts. This is another strategy one can use.} 
\par The situation gets complicated when it comes to a singular $J$ since the cosquare $J^{-T}\!J$ is no longer well defined. However the theory of \textit{matrix pencils} is broad enough to cover such cases. The \textit{Kronecker structure} \cite{gantmacher1964theory,kagstrom1982matrix,kublanovskaya1966method,van1979computation} of the matrix pencil $J - \lambda J^T$ provides a generalization of the eigenstructure of the cosquare, enabling one to compute the solution set of $X^T\!J+JX=0$ in a similar manner. By revealing the structure, our approach is advantageous since it is not limited to the complex case. Moreover it can be applied to the situation when we have an involution other than the complex conjugation. 

\par The outline of the paper is the following: In Section \ref{sec:background} we provide background for solving the equations $X^*\!J\pm JX=0$ and investigate some small sized groups $\{G:G^*\!JG = J\}$. The main tools developed are Theorems \ref{thm:complexTsolcosol}, \ref{thm:complexHsolcosol}, and \ref{thm:realsolcosol} which provide an explicit way to construct a basis of the solution set of $X^*\!J\pm JX=0$. We also demonstrate the analysis on structures of $2\times 2$ matrix groups $G_J^\mathbb{R}$ in Section \ref{sec:2x2structure}. In Sections \ref{sec:complexTdimcount} to \ref{sec:realdimcount} we determine precise bases of the solution sets to $X^*\!J\pm JX=0$. In Section \ref{sec:numerical} we discuss numerical details of the computation and visualization of automorphism groups and their tangent spaces.

\section{Background}\label{sec:background}

\subsection{Warm up : Centralizer and Jordan (generalized eigenvector) chains}\label{sec:centralizer}
Given an $n\times n$ matrix $A$, the set of all matrices that commute with $A$ is the \textit{centralizer} (in operator theory, \textit{commutant}) of $A$, denoted by $\cent(A)$. Typically $\cent(A)$ can be obtained directly from the Jordan canonical form of $A$ \cite{arnold1971matrices,gantmacher1964theory}.

\par One way of describing $\cent(A)$ is using Jordan chains (i.e., generalized eigenvector chains) of $A$. Fix an eigenvalue $\lambda$ and let $r_1, \dots, r_k$ be the sizes of the Jordan blocks. Select matrices $W_1, \dots, W_k$ with their columns filled with Jordan chains so that $AW_j = W_j J_{r_j}^\lambda$ holds for the $r_j\times r_j$ Jordan block $J_{r_j}^\lambda$. Similarly, choose Jordan chain matrices $P_1, \dots, P_k$ of $A^T$. Then, $\cent(A)$ is constructed as follows.
\begin{definition}\label{def:backwardsid}
Denote the $n\times n$ backwards identity matrix by $E_n$. For $j\leq \min(s, t)$ define $E_j^{s, t}$, the $s\times t$ matrix with zeros except for its upper left $j\times j$ corner being $E_j$.  
\end{definition}

\begin{lemma}\label{lem:basiscent}
Let $W\in \mathbb{F}^{n\times a}, P\in\mathbb{F}^{n\times b}$ be Jordan chain matrices corresponding to the eigenvalue $\lambda$ of $A$ and $A^T$, respectively. (They may not correspond to the same Jordan block.) Then the collection of the matrices ($m = \min(a, b)$)
\begin{equation}\label{eq:basiscent}
    W E_1^{a,b} P^T, W E_2^{a,b} P^T, \dots, W E_m^{a,b} P^T,
\end{equation}
for all $\lambda$ (and all combinations of $W, P$) span $\cent(A)$. 
\end{lemma}
\begin{proof}
This is a simple variant of the description in Chapter \RN{8} of \cite{gantmacher1964theory}. 
\end{proof}

%\par Lemma \ref{lem:basiscent} is a matrix-wise expression of the following description: For generalized eigenvector chains $w_1, \dots, w_{n_1}$ and $p_1, \dots, p_{n_2}$ of $A$ and $A^T$, we create a convolution-like sum $X_j = w_1p_j^T+w_2p_{j-1}^T+\dots +w_j p_1^T$. Observe $(A-\lambda I)X_j = X_{j-1}$ and $X_j (A-\lambda I) = X_{j-1}$ so that $X_j\in\cent(A)$. The matrices $X_1, \dots, X_m$ are exactly the matrices \eqref{eq:basiscent}. 
Counting the total number of matrices of the form \eqref{eq:basiscent} we obtain the dimension of the centralizer. (See, for instance, \cite{arnold1971matrices,gantmacher1964theory,gohberg2006invariant,macduffee1933}.)

%%%% MAYBE in algorithmic consideration part?
%\par Lemma \ref{lem:basiscent} can also be extended to the cases when columns of $W$ and $U$ are not generalized eigenvector chains but a basis of a single invariant subspaces. Let $W'$ and $U'$ be $n\times n_1$ and $n\times n_2$ matrices with columns being the basis of the invariant subspaces of eigenvalue $\lambda$. Then $AW' = W'M$, $A^TU' = U'N$ for some $M$ and $N$ with sizes $n_1\times n_1$ and $n_2\times n_2$ each. Instead of \eqref{eq:basiscent}, we can define $W'E_j'{U'}^T$ where $E_1', \dots, E_m'$ are matrices such that $ME_j' = {E_j'}^TN$. 

\begin{corollary}\label{cor:centdimcount}
Let $A$ be a square matrix with the Jordan form $A = \bigoplus_{j=1}^n J_{r_j}^{\lambda_j}$. 
\begin{equation*}
    \dim(\cent(A)) = \sum_{j=1}^n r_j + \sum_{\lambda_j=\lambda_k} \min(2r_j, 2r_k).
\end{equation*}
\end{corollary}

\subsection{Basics : solution and cosolution}\label{sec:solandcosol}
\par We start by defining some basic notation. We label the solution sets to the equations $X^*\!J\pm JX=0$ as follows.

\begin{definition}\label{def:solcosol}
For a given square matrix $J$ define four solution sets,
\begin{align*}
    \sol(J) &:= \{X : X^T\!J+JX=0\}, \hspace{.7cm} \cosol(J):=\{X:X^T\!J-JX=0\} \\
    \sol^H(J) &:= \{X : X^H\!J+JX=0\}, \hspace{.7cm} \cosol^H(J):= \{X:X^H\!J-JX=0\}.
\end{align*}
We call $\sol(J)$ the \textit{solution} and $\cosol(J)$ the \textit{cosolution}. 
\end{definition}

\begin{remark}\label{rem:solHreal}
At first glance, one might believe all four of these sets are complex vector spaces (implying, for example, that multiplying by a complex scalar is a closed operation) given that $J$ is complex, since they seem to be the homogeneous solutions of linear equations. But a closer inspection reveals that the ``$H$" spaces are not complex vector spaces since, for example, if $X$ is a solution, $iX$ need not be. They are, however, real vector spaces. The matrix transposition is an analytic map and $\sol(J)$ is a complex vector space just like the Lie group $G_J$ is a complex manifold. On the other hand, the Lie group $G_J^H$ is not a complex group since the conjugate transposition is not analytic. As a consequence, $\sol^H(J)$ and $\cosol^H(J)$ are real (but not complex) vector spaces. For example, $\text{U}(n)$ is a real Lie group and $\mathfrak{u}(n)$ (skew-Hermitian matrices) is a real vector space. $\text{O}(n, \mathbb{C})$ is a complex Lie group and $\mathfrak{o}(n, \mathbb{C})$ (complex skew-symmetric matrices) is a complex vector space.
\end{remark}

\subsection{Nonsingular $J$ and the cosquare}\label{sec:nonsingulartheory}
Assume for a moment that our $J$ is complex nonsingular, because many key intuitions arise when we study nonsingular $J$. Singular $J$ will be discussed later in Section \ref{sec:singularJbackground}. One can also consider real or quaternionic $J$ with a simple modification of what we describe in this section.

\par An important matrix related to $\sol(J)$ and $\cosol(J)$ is the \textit{cosquare}\footnote{The definition of the cosquare of $J$ is sometimes different. In fact, the four matrices $JJ^{-*}$, $J^*J^{-1}$, $J^{-1}J^*$, $J^{-*}J$ are all very much alike. Authors usually select one of them for their own needs. For example in \cite{horn2012matrix} the cosquare is given as $J^{-*}J$ and in \cite{taussky1979some} it is defined as $J^{-1}J^*$.} $J^{-T}\!J$ of $J$. 
\begin{definition}
Given a square nonsingular matrix $A$, the matrix $A^{-T}\!A$ is called the \textit{cosquare} of $A$ and denote it by $\cosq(A)$. If $A$ is a complex or quaternionic matrix define the \textit{$H$-cosquare} of $A$ by $A^{-H}\!A$ and denoted it by $\cosq^H(A)$. 
\end{definition} 

\par Obviously $C = J^{-T}\!J \sim JJ^{-T}$ ($\sim$ stands for matrix similarity) since $AB\sim BA$ for any invertible $A, B$. ($A^{-1}(AB)A = BA$.) Using that any matrix is similar to its transpose (e.g., see \cite[Thm 3.2.3.1]{horn2012matrix}), $JJ^{-T}\sim J^{-1}J^T=C^{-1}$. Lemma \ref{lem:cosqeigenstructure} follows. 
\begin{lemma}\label{lem:cosqeigenstructure}
For a given nonsingular $J$, the cosquare $C$ is similar to its inverse. Thus, $C, C^T, C^{-1}, C^{-T}$ are all similar to each other. Moreover, the $H$-cosquare $D$ is similar to its conjugate inverse, $\bar{D}^{-1}$.
\end{lemma}

\par From Lemma \ref{lem:cosqeigenstructure} the eigenstructures (sizes and numbers of Jordan blocks) of $C, C^T, C^{-1}, C^{-T}$ are all identical. Let us select four Jordan chain\footnote{We state only the results for the regular transpose but the cases $X^H\!J\pm JX=0$ are similar.} matrices $W, U, P, Q$ of $C, C^{-1}, C^T, C^{-T}$, respectively, all corresponding to the same Jordan block with eigenvalue $\lambda$. The columns of $W$, for example, $w_1, \dots, w_r$ satisfy the relationships $(C-\lambda I)w_j = w_{j-1}$ with $w_0=0$, namely,
\begin{equation*}
    \begin{bmatrix} & & \\ &\hspace{0.2cm} C\hspace{0.2cm}\, & \\ & &\end{bmatrix}
    \begin{bmatrix} \vert & &\vert \\ w_1 & \cdots & w_r \\ \vert & & \vert \end{bmatrix} = \begin{bmatrix} \vert & &\vert \\ w_1 & \cdots & w_r \\ \vert & & \vert \end{bmatrix}
    \begin{bmatrix}\lambda & 1 & 0 \\  &\ddots & 1\\ & & \lambda\end{bmatrix}.
\end{equation*}
By definition, the four Jordan chains satisfy (with the Jordan block $J^\lambda$)
\begin{equation}\label{eq:WUPQdefinition}
CW = WJ^\lambda,\hspace{0.5cm}C^{-1}U = UJ^\lambda,\hspace{0.5cm}C^TP = PJ^\lambda, \hspace{0.5cm} C^{-T}Q = QJ^\lambda.
\end{equation}
%%% RESUME FROM HERE
\par Interestingly, $J^{-T}\!JW = W J^\lambda$ is equivalent to $JJ^{-T}\big(JW\big) = \big(JW\big)J^\lambda$, which makes the matrix $JW$ eligible as a choice of $Q$. Moreover, $J^{-T}\!JW = WJ^\lambda$ is equivalent to $JJ^{-T}\big(J^TW\big) = \big(J^TW\big)J^\lambda$, which makes $J^TW$ eligible as a choice of $Q$. We deduce the following relationships. (Denote the set of all possible choices of a Jordan chain $W$ by $\{W\}$, and similarly the other chains.)
\begin{equation}\label{eq:UWPQrelationship}
    \begin{array}{cc}
         JW, J^TW \in \{Q\}, & \hspace{1cm}JU, J^TU \in \{P\}, \\
         J^{-1}P, J^{-T}P \in \{U\}, &  \hspace{1cm}J^{-1}Q, J^{-T}Q \in\{W\}.
    \end{array}
\end{equation}

Then, $\sol(J)$, $\cosol(J)$ and $\cent(\cosq(J))$ have the following important property. 
\begin{theorem}\label{thm:solandcosol}
Let $Z\in\cent(\cosq(J))$. Then the following holds:
\begin{gather}\label{eq:sol}
    Z - J^{-1}Z^TJ = Z-J^{-T}Z^TJ^T\in \sol(J),\\
    Z + J^{-1}Z^TJ = Z+J^{-T}Z^TJ^T\in \cosol(J). \label{eq:cosol}
\end{gather}
Similarly, we have $Z-J^{-1}Z^HJ= Z-J^{-H}Z^HJ^H \in \sol^H(J)$ and $Z+J^{-1}Z^HJ=Z+J^{-H}Z^HJ^H \in \cosol^H(J)$ for $Z\in\cent(\cosq^H(J))$. 
\end{theorem}
\begin{proof}
One needs only a simple algebraic manipulation to see this. Since $Z\in\cent(\cosq(J))$, we have $ZJ^{-T}J = J^{-T}JZ$. Then, for $X := Z-J^{-1}Z^TJ$, 
\begin{align*}
    X^T\!J+JX = Z^TJ - &J^TZJ^{-T}J + JZ - Z^TJ \\
    &=  -J^TJ^{-T}JZ + JZ = -JZ + JZ = 0.
\end{align*}
Similarly one can obtain all other results.
\end{proof}

\par Any $X$ in $\sol(J)\cap\cosol(J)$ satisfies $X^T\!J=JX=0$ and the invertibility of $J$ implies $X=0$. Since the intersection is trivial, one can compute for any $Z \in \cent(\cosq(J))$ unique $X\in\sol(J)$, $Y\in\cosol(J)$ such that $Z = X + Y$.
\begin{corollary}\label{cor:solcosoldirectsum}
For a nonsingular $J$ we have
\begin{equation*}
    \sol(J) \oplus \cosol(J) = \cent(\cosq(J)). 
\end{equation*}
For the conjugate transpose we have $\sol^H(J) \oplus \cosol^H(J) = \cent(\cosq^H(J))$.
\end{corollary} 
The two maps $Z\mapsto (Z-J^{-1}Z^TJ)/2$ and $Z \mapsto (Z+J^{-1}Z^TJ)/2$ serve as projections of $\cent(\cosq(J))$ down to $\sol(J)$ and $\cosol(J)$. Thinking of $J^{-1}Z^TJ$ as a special transpose, one can treat $\sol(J)$ and $\cosol(J)$ as analogs of skew-symmetric and symmetric matrices, respectively. (In fact, for $J=I$ this is exactly the case.) 

\subsection{The automorphism group when $J$ is a small matrix}

\subsubsection{Generic $4\times 4$ real $J$}\label{sec:4x4generic}

\par It is certainly helpful to work out some small cases to get a grasp on the structures of the automorphism groups. For small sized matrices, Lemma \ref{lem:cosqeigenstructure} can be directly applied to determine the eigenvalue characteristic of the cosquare. Figure \ref{fig:4x4plots} in the introduction contains one such example. The four generic eigenvalue profiles (in red) in Figure \ref{fig:4x4plots} are the four possible scenarios by applying Lemma \ref{lem:cosqeigenstructure} to the real $4\times 4$ cosquare $J^{-T}\!J$. 

\par If there are no double eigenvalues as in these cases, we have a simpler situation. Using Theorem 2.1.(d) of \cite{horn2007canonical} it follows that a given generic $J\in\mathbb{R}^{4\times 4}$ is congruent to a block combination of $2\times 2$ type (\rn{2}) matrices, $4\times 4$ type (\rn{2}\pp) matrices, and $2\times 2$ type (\rn{3}\pp) matrices in the Theorem. Type (\rn{2}) matrices have their cosquares with two real eigenvalues $(\lambda, 1/\lambda)$, type (\rn{2}\pp) matrices' cosquares contain four eigenvalues $(\lambda, \bar\lambda, 1/\lambda, 1/\bar\lambda), \lambda\in\mathbb{C}\backslash\mathbb{R}$, and type (\rn{3}\pp) matrices have their cosquares with pairs of unit eigenvalues $(\alpha, \bar\alpha), |\alpha|=1$. For example, a matrix $J$ with the top left eigenvalue profile of Figure \ref{fig:4x4plots} is congruent to a combination of two (\rn{3}\pp) blocks with distinct unit complex eigenvalues. 

\par Then by computing the block solutions (since there are no interactions between blocks) as in \cite{de2011trans} but with real matrices, one realizes that the solution sets have log-level eigenvalue profiles. Namely, for type (\rn{2}) $J$ block with the cosquare $X\begin{bsmallmatrix}\lambda & \\ & 1/\lambda\end{bsmallmatrix}X^{-1}$, we have the solution set $\{X\begin{bsmallmatrix}t & \\ & -t\end{bsmallmatrix}X^{-1}: t\in\mathbb{R}\}$. Similarly for a type (\rn{2}\pp) $J$ block we have $\sol(J)$ similar to the collection of $\text{diag}(t, -t, \bar{t}, -\bar{t}), t\in\mathbb{C}\backslash\mathbb{R}$, and for a type (\rn{3}\pp) $J$ block we have $\sol(J)$ similar to all $\text{diag}(it, -it), t\in\mathbb{R}$. 

\par Then we use the exponential map (see Section \ref{sec:exp}) to obtain $G_J^\mathbb{R}$. The eigenvalue profiles of the matrices in $G_J^\mathbb{R}$ become the eigenvalue profile of $\cosq(J)$. This could already be seen from the fact that $\cosq(J)=J^{-T}\!J$ itself belongs to $G_J^\mathbb{R}$. For instance, in the top left case of Figure \ref{fig:4x4plots} we have a $J$ such that the eigenvalues of $\cosq(J)$ are $(\alpha_1, \bar\alpha_1, \alpha_2, \bar\alpha_2)$ with $\alpha_1, \alpha_2$ both in the unit complex circle. The group $G_J^\mathbb{R}$ then is similar to the collection of all $(\beta, \bar\beta, \gamma, \bar\gamma)$ where $\beta, \gamma$ are on the complex unit circle. Thus, $G_J^\mathbb{R}$ is diffeomorphic to $S^1\times S^1$.

\par Figure \ref{fig:4x4plots} is the visualizations of the identity components of $G_J^\mathbb{R}$ for the four possible cases: circle $\times$ circle (top left), $\mathbb{C}\backslash\{0\}$ (top right), hyperbola $\times$ circle (bottom left) and hyperbola $\times$ hyperbola (bottom right). They are also the eigenstructure of the cosquare in each case. Furthermore, since the set of real eigenpairs $(\lambda, 1/\lambda)$ are disconnected for positive and negative $\lambda$'s, there are two such components for the bottom left case. Similarly for the bottom right case the group $G_J^\mathbb{R}$ contains four isomorphic copies of the identity component.

\subsubsection{$2\times 2$ real $J$}\label{sec:2x2structure}
We consider $J\in\mathbb{R}^{2\times 2}$. For a given $J$, let $J=S+A$ be the decomposition of $J$ into its symmetric part $S$ and skew-symmetric part $A$. It is useful that the signature of $S$ and the rank of $A$ are both invariant under congruence transformations on $J$. Carefully classifying the group structure for all real $2\times 2$ cases we obtain the following Table \ref{tab:2x2cases}.  

\begin{table}[h]
\begin{center}
{\tabulinesep=0.5mm
\small
\centering
\title{Groups $G_J^\mathbb{R} = \{G:G^T\!JG=J\}$ up to similarity transformation, $J\in\mathbb{R}^{2\times 2}$}
\begin{tabu}{|c|c|c|c|c|c|}
\hline
 & \makecell{$J$ up to \\congruence} & \makecell{Eigenvalues\\of $\cosq(J)$} & \makecell{Signature$(S)$, \\ $\text{rank}(A)$} & \makecell{Group structure\\  (up to conjugation)} & Dim\\ \hline
\circled{1} & \makecell{Generic \\ $\text{Sign}(S)\!=\!(1,1)$} & \makecell{$\lambda, \frac{1}{\lambda}$ \\ $\lambda\in\mathbb{R}$} & $(1, 1), 2$ & \makecell{$\Big\{\begin{bsmallmatrix}x & 0 \\ 0 & \frac{1}{x}\end{bsmallmatrix}: x\in\mathbb{R}\backslash\{0\}\Big\}$ \\ (Hyperbola)} & 1\\ \hline
\circled{2} & \makecell{Generic \\ $\text{Sign}(S)\!=\!(2,0)$} & \makecell{$\lambda, \bar\lambda$ \\ $|\lambda|=1$} & $(2, 0),2$ & \makecell{$\{\begin{bsmallmatrix}c & s \\ \mis s & c\end{bsmallmatrix}: c^2+s^2=1\}$ \\ (Circle)} & 1 \\ \hline
\circled{3} & $I_2$ & $J_1^1\oplus J_1^1$ & $(2, 0), 0$ & \makecell{$\text{O}(2)$ \\ (Two circles)} & 1\\ \hline
\circled{4} & $I_{1,1}=$\scalebox{0.8}[0.8]{$\begin{bmatrix}1 & 0 \\ 0 & \mi 1\end{bmatrix}$} & $J_1^1\oplus J_1^1$ & $(1, 1),0$ & \makecell{$\text{O}(1,1)$\\ (Two hyperbolae)} & 1 \\ \hline
\circled{5} & \scalebox{0.8}[0.8]{$\begin{bmatrix}\,\,0 & 1 \\ \,\,\mi 1 & 0\end{bmatrix}$} & $J_1^{-1}\oplus J_1^{-1}$ & $(0, 0), 2$ & $\text{SL}(2, \mathbb{R}) = \text{Sp}(2, \mathbb{R})$ & 3 \\ \hline
\circled{6} & \scalebox{0.8}[0.8]{$\begin{bmatrix}0 & \mi 1 \\ 1 & 1\end{bmatrix}$} & $J_2^{-1}$ & $(1, 0), 2$ & \makecell{$\{\pm \begin{bsmallmatrix}1 & x \\ 0 & 1\end{bsmallmatrix}: x\in\mathbb{R}\}$ \\ (Two Real lines)} & 1\\ \hline
\circled{7} & \scalebox{0.8}[0.8]{$\begin{bmatrix} 0 & 1 \\ 0 & 0 \end{bmatrix}$} & - & $(1, 1), 2$ & \makecell{$\Big\{\begin{bsmallmatrix}x & 0 \\ 0 & \frac{1}{x}\end{bsmallmatrix}: x\in\mathbb{R}\backslash\{0\}\Big\}$ \\ (Hyperbola)} & 1\\ \hline
\circled{8} & \scalebox{0.8}[0.8]{$\begin{bmatrix} 1 & 0 \\ 0 & 0 \end{bmatrix}$} & - & $(1, 0), 0$ & \makecell{$\{\begin{bsmallmatrix}\pm 1 & 0 \\ x & y\end{bsmallmatrix}: x,y\in\mathbb{R},y\neq 0\}$ \\ $\{-1, 1\}\times \text{Aff}(1, \mathbb{R})$}& 2\\ \hline
\circled{9} & Zero matrix & - & $(0, 0),0$ & $\text{GL}(2, \mathbb{R})$ & 4\\ \hline
\end{tabu}}
\end{center}
\caption{Nine types of the real group $G_J^\mathbb{R}$ for a $2\times 2$ real $J$ (and their dimensions). Note that the signature $(a, b)$ represents both $(a, b)$ and $(b, a)$.}
\label{tab:2x2cases}
\end{table}

The nine cases also have the closure graph (sometimes called bundle stratification or closure hierarchy) as illustrated in Figure \ref{fig:closure}. In short, a closure graph is a Hasse diagram with a partial order $\preceq $, where $x\preceq y$ if $x$ is contained in the closure of $y$. See \cite{dmytryshyn2015change,futorny2014change} for more on the closure graphs of matrix groups. After working out the real case for $2\times 2$ and $3\times 3$ matrices (for $3\times 3$ case, see \cite{sungwoo2021}), we were happy to find that the slightly simpler complex cases can already be found in \cite{dmytryshyn2015change}. Figure \ref{fig:closure} is analogous to the relationships (e.g., Figures 1 and 2) in \cite{dmytryshyn2015change} since congruent $J$'s deduce similar automorphism groups. 

\begin{figure}[ht]
    \centering
    \includegraphics[width=4.9in]{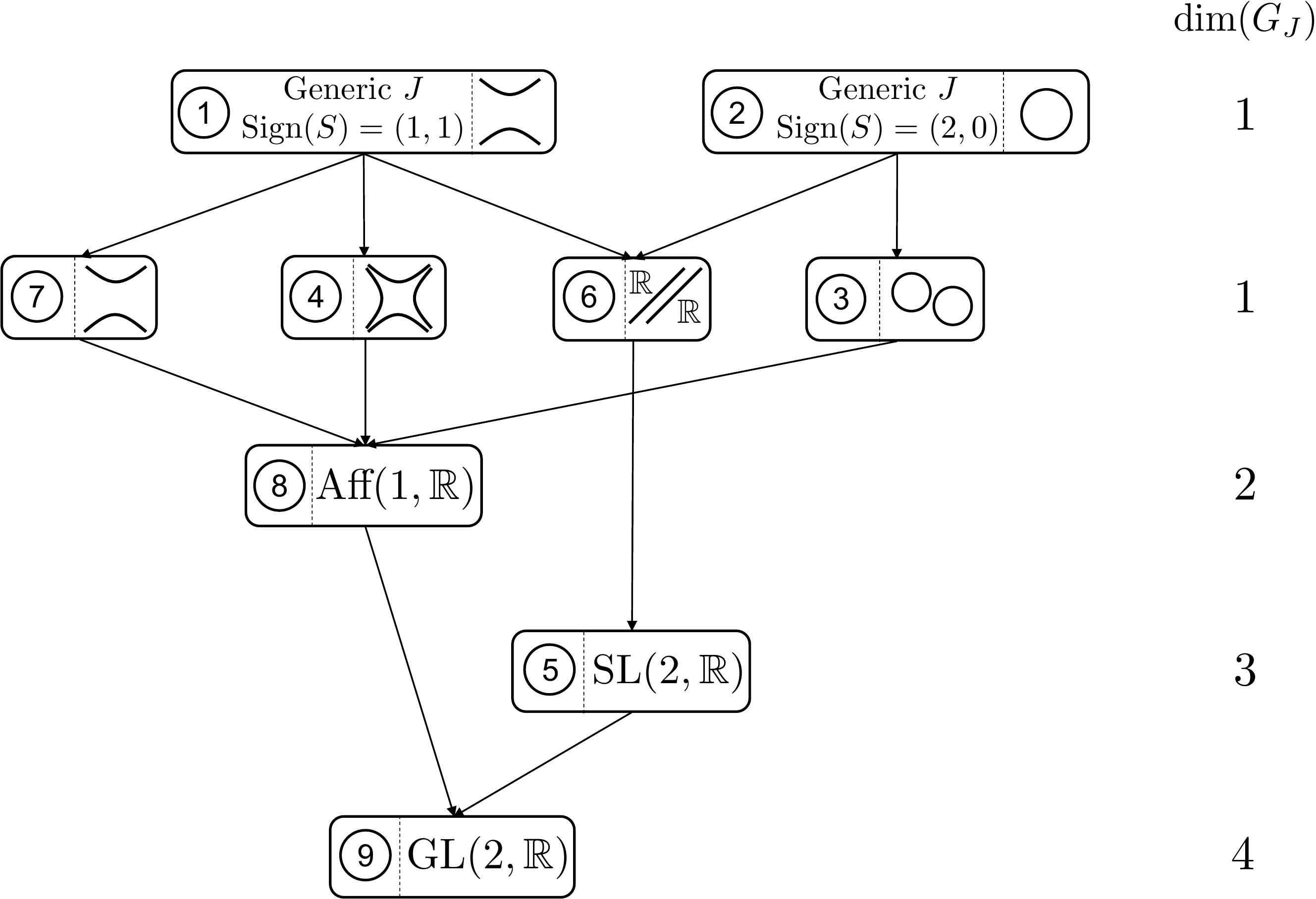}
    \caption{The closure relationship of real $2\times 2$ matrices classified by the structure of $G_J^\mathbb{R}$. Details of each cell can be found in Table \ref{tab:2x2cases}. Each cell contains a simple illustration of the group structure such as circle, hyperbola and more. \textcircled{ \small{$\!\!$A}}$\to$\textcircled{\small B} means that ``the closure of A contains B".}
    \label{fig:closure}
\end{figure}

\subsection{Main results}
By Theorem \ref{thm:solandcosol}, it is possible to construct a basis and compute the dimension of $\sol(J)$. Begin with a basis of $\cent(\cosq(J))$ and use the maps \eqref{eq:sol}, \eqref{eq:cosol} to obtain the sets that contain bases of $\sol(J)$ and $\cosol(J)$. More explicitly, if we select a basis $\cup \{W E_k^{a,b} P^T\}$ of $\cent(C)$ and apply \eqref{eq:sol} with a substitution of $J^{-1}P=U$, we obtain $\cup\{ WE_k^{a,b}U^TJ^T - UE_k^{b,a}W^TJ \}$. This is the projection of a basis of $\cent(\cosq(J))$ onto $\sol(J)$. In Sections \ref{sec:complexTdimcount} to \ref{sec:realdimcount} we will rule out linearly dependent matrices from the projections to determine precise bases of $\sol(J)$ and $\cosol(J)$. 

\begin{definition}\label{def:XTYT}
For matrices $J\in\mathbb{C}^{n\times n}$, $W\in\mathbb{C}^{n\times a}$, $U\in\mathbb{C}^{n\times b}$ define four matrices $X_T$, $Y_T$, $X_H$, $Y_H$ as follows: ($k=1,\dots, \min(a, b)$)
\begin{gather}\label{eq:complexTXT}
    X_T(k, J, W, U) = WE_k^{a, b}U^TJ^{T} - U E_k^{b, a}W^TJ, \\
    Y_T(k, J, W, U) = WE_k^{a, b}U^TJ^{T} + U E_k^{b, a}W^TJ, \\
    X_H(k, J, W, U) = WE_k^{a, b}U^HJ^{H} - U E_k^{b, a}W^HJ, \\
    Y_H(k, J, W, U) = WE_k^{a, b}U^HJ^{H} + U E_k^{b, a}W^HJ.
\end{gather}
\end{definition}
\begin{theorem}[Computing $\sol(J)$, $\cosol(J)$ for a complex $J$]\label{thm:complexTsolcosol}
Given a nonsingular $J\in\mathbb{C}^{n\times n}$, let $\Lambda(C)$ be the set of eigenvalues of $C=\cosq(J)$. Let $r_1, \dots, r_m$ be the sizes of the Jordan blocks of $C$ corresponding to $\lambda\in\Lambda(C)$. For $j=1, \dots, m$, select the Jordan chain matrices $W_\lambda^{(j)}, U_\lambda^{(j)}\in\mathbb{C}^{n\times r_j}$ of $C,C^{-1}$ respectively, corresponding to the eigenvalue $\lambda$. Then, the sets 
\begin{gather}\label{eq:complexTsol}
    B_T^+(J) = \bigcup_{\lambda\in\Lambda(C)}\bigg(\bigcup_{1\le s, t\le m} \bigcup_{k=1}^{\min(r_s, r_t)}\Big\{X_T(k, J, W_\lambda^{(s)}, U_\lambda^{(t)})\Big\}\bigg), \\
    B_T^-(J) = \bigcup_{\lambda\in\Lambda(C)}\bigg(\bigcup_{1\le s, t\le m} \bigcup_{k=1}^{\min(r_s, r_t)}\Big\{Y_T(k, J, W_\lambda^{(s)}, U_\lambda^{(t)})\Big\}\bigg),\label{eq:complexTcosol}
\end{gather} 
span the complex vector spaces $\sol(J)$ and $\cosol(J)$, respectively. 
\end{theorem}
\begin{proof}
From Lemma \ref{lem:basiscent} and \eqref{eq:UWPQrelationship} we have a basis of $\cent(C)$, 
\begin{equation*}
    \bigcup_{\lambda\in\Lambda(C)}\bigg(\bigcup_{1\le s, t\le m} \bigcup_{k=1}^{\min(r_s, r_t)}\Big\{W_\lambda^{(s)}E_k^{r_s, r_t}{U_\lambda^{(t)}}^TJ^T\Big\}\bigg).
\end{equation*}
Since $\sol(J)\oplus\cosol(J) = \cent(C)$ we apply the projection map $\cent(C)\to\sol(J)$ in Theorem \ref{thm:solandcosol} to the above basis and obtain the set 
\begin{equation*}
    \bigcup_{\lambda\in\Lambda(C)}\bigg(\bigcup_{1\le s, t\le m} \bigcup_{k=1}^{\min(r_s, r_t)}\Big\{W_\lambda^{(s)}E_k^{r_s, r_t}{U_\lambda^{(t)}}^TJ^T - U_\lambda^{(t)}E_k^{r_t, r_s}W_\lambda^{(s)}J\Big\}\bigg),
\end{equation*}
which is exactly $B_T^+(J)$. Note that $B_T^+(J)$ spans $\sol(J)$ since it is a projected basis of $\cent(C)$. (Nonetheless, $B_T^+(J)$ may not be a linearly independent.) The set $B_T^-(J)$ could also be obtained and it spans $\cosol(J)$.
\end{proof}

\par For the equation $X^H\!J+JX=0$, the solution set $\sol^H(J)$ is similarly obtained. However as discussed in Remark \ref{rem:solHreal}, the solution and the cosolution sets are real vector spaces. Then $\sol^H(J)$ and $\cosol^H(J)$ are spanned by the following matrices.  

\begin{theorem}[Computing $\sol^H(J), \cosol^H(J)$ for a complex $J$]\label{thm:complexHsolcosol}
Given a nonsingular $J\in\mathbb{C}^{n\times n}$, let $\Lambda(C)$ be the set of eigenvalues of $C=\cosq^H(J)$. Let $r_1, \dots, r_m$ be the sizes of the Jordan blocks of $C$ corresponding to $\lambda\in\Lambda(C)$. For $j=1, \dots, m$, select Jordan chain matrices $W_\lambda^{(j)}, U_{\bar\lambda}^{(j)}\in\mathbb{C}^{n\times r_j}$ of $C, C^{-1}$, corresponding to the eigenvalues $\lambda$ and $\bar\lambda$, respectively. Then, the sets
\begin{gather}\label{eq:complexHsol}
    B_H^+(J) = \bigcup_{\lambda\in\Lambda(C)}\bigcup_{1\le s, t\le m} \hspace{-0.2cm}\bigcup_{k=1}^{\min(r_s, r_t)} \hspace{-0.2cm} \Big\{X_H(k, J, W_\lambda^{(s)}, U_{\bar\lambda}^{(t)}), i Y_H(k, J, W_\lambda^{(s)}, U_{\bar\lambda}^{(t)}) \Big\}, \\
    B_H^-(J) = \bigcup_{\lambda\in\Lambda(C)}\bigcup_{1\le s, t\le m} \hspace{-0.2cm}\bigcup_{k=1}^{\min(r_s, r_t)} \hspace{-0.2cm}\Big\{iX_H(k, J, W_\lambda^{(s)}, U_{\bar\lambda}^{(t)}), Y_H(k, J, W_\lambda^{(s)}, U_{\bar\lambda}^{(t)}) \Big\},\label{eq:complexHcosol}
\end{gather}
span the real vector spaces $\sol^H(J)$ and $\cosol^H(J)$, respectively.
\end{theorem}
\begin{proof}
The proof is nearly identical to the proof of Theorem \ref{thm:complexTsolcosol}, except for that one has to consider a real basis of $\cent(C)$ by separating $WEP^T$ and $iWEP^T$ in Lemma \ref{lem:basiscent}.  
\end{proof}

Now we consider the real case.\footnote{The quaternionic version of $\sol(J)$, $\cosol(J)$ can be obtained in a similar manner but in this work we will not discuss the details.} Define auxiliary matrices as in Definition \ref{def:XTYT}. 
\begin{definition}[The ``realify" map]
Let $A$ be an $n\times m$ complex matrix. Define the \textit{realify} map $\realify{\,\,\cdot\,\,}:\mathbb{C}^{n\times m}\to \mathbb{R}^{2n\times 2m}$ as the following block matrices:
\begin{equation*}
    \realify{A} = \begin{bmatrix}\begin{array}{rr} \text{re}(A) & \text{im}(A) \\ -\text{im}(A) & \text{re}(A) \end{array}\end{bmatrix}.
\end{equation*}
\end{definition}

\begin{definition}\label{def:realXTYT}
Let $J$ be a nonsingular real $n\times n$ matrix. For $W\in\mathbb{R}^{n\times 2a}$, $U\in\mathbb{R}^{n\times 2b}$, and an integer $k\leq \min(a, b)$ we define the following matrices $X_\mathbb{R}$, $Y_\mathbb{R}$:
\begin{gather*}
    X_\mathbb{R}(k, J, W, U) = W(E_2\otimes E_{k}^{a, b})U^TJ^{T} - U(E_2\otimes E_{k}^{b, a}) W^TJ,\\
    Y_\mathbb{R}(k, J, W, U) = W(E_2\otimes E_{k}^{a, b})U^TJ^{T} + U(E_2\otimes E_{k}^{b, a}) W^TJ.
\end{gather*}
\end{definition}

The solution and cosolution sets $\sol(J), \cosol(J)\subset\mathbb{R}^{n\times n}$ for the real case follows.
\begin{theorem}[Computing real $\sol(J), \cosol(J)$ for a real $J$]\label{thm:realsolcosol}
Given a nonsingular $J\in\mathbb{R}^{n\times n}$, let $\Lambda(C)$ be the eigenvalues of $C=\cosq(J)$. For all real eigenvalues in $\Lambda(C)$, proceed as in Theorem \ref{thm:complexTsolcosol} with real Jordan chain matrices. Denote the resulting sets \eqref{eq:complexTsol}, \eqref{eq:complexTcosol} by $B_T^+(J,\mathbb{R})$, $B_T^-(J,\mathbb{R})$. For all other $\lambda\in\Lambda(C)\backslash\mathbb{R}$, let $r_1, \dots, r_m$ be the sizes of Jordan blocks of $C$ corresponding to $\lambda$. For $j=1, \dots, m$, select real Jordan chain matrices $W_\lambda^{(j)}, U_\lambda^{(j)}\in\mathbb{R}^{n\times 2r_j}$ corresponding to the pair $\lambda, \bar\lambda$ such that $C W_\lambda^{(j)} = W_\lambda^{(j)}\realify{J_{r_j}^\lambda}$ and $C^{-1}U_\lambda^{(j)} = U_\lambda^{(j)}\realify{J_{r_j}^\lambda}$. Then, the sets 
\begin{equation}\label{eq:realsol}
    B_\mathbb{R}^+(J) = B_T^+(J,\mathbb{R}) \cup \bigcup_{\lambda\in\Lambda(C)\backslash\mathbb{R}}\bigg(\bigcup_{1\le s, t\le m} \bigcup_{k=1}^{\min(r_s, r_t)} \{X_{\mathbb{R}}(k, J, W_\lambda^{(s)}, U_\lambda^{(t)})\}\bigg),
\end{equation}
\begin{equation}\label{eq:realcosol}
    B_\mathbb{R}^-(J) = B_T^-(J,\mathbb{R}) \cup \bigcup_{\lambda\in\Lambda(C)\backslash\mathbb{R}}\bigg(\bigcup_{1\le s, t\le m} \bigcup_{k=1}^{\min(r_s, r_t)} \{Y_{\mathbb{R}}(k, J, W_\lambda^{(s)}, U_\lambda^{(t)}) \}\bigg),
\end{equation}
span the real vector spaces $\sol(J)$ and $\cosol(J)$, respectively. 
\end{theorem}
\begin{proof}
This time we need a basis of $\cent(C)$ that lies inside $\mathbb{R}^{n\times n}$, which is a modification of Lemma \ref{lem:basiscent}. Instead of using a complex Jordan matrix to derive a centralizer basis, we use a real modified Jordan form of $C$. Since eigenvalues $\lambda, \bar\lambda$ exist together (with the same sized Jordan blocks) the two Jordan blocks $J_{r_j}^\lambda$ and $J_{r_j}^{\bar\lambda}$ together could be expressed as $\realify{J_{r_j}^\lambda}$. Then, with real Jordan chain matrices $W\in\mathbb{R}^{n\times 2r_s}, P\in\mathbb{R}^{n\times 2r_t}$ such that $CW = W\realify{J_{r_s}^\lambda}$ and $C^TP = P\realify{J_{r_t}^\lambda}$ (one can obtain such $W, P$ by realifying complex Jordan chain matrices of $\lambda$) it could be deduced that the collection of $W(E_2\otimes E_k^{r_s, r_t})P^T$ is a basis of $\cent(C)$. Then, we proceed as in Theorems \ref{thm:complexTsolcosol} to obtain the sets $B_\mathbb{R}^+(J), B_\mathbb{R}^-(J)$.
\end{proof}

\par The following proposition states that the selection of the Jordan chains in Theorems \ref{thm:complexTsolcosol}, \ref{thm:complexHsolcosol}, and \ref{thm:realsolcosol} does not affect the resulting vector spaces. 

\begin{proposition}[equivalence of Jordan chains]\label{thm:choiceofWU}
In Theorems \ref{thm:complexTsolcosol}, \ref{thm:complexHsolcosol}, and \ref{thm:realsolcosol}, the choices of the Jordan chains $W_\lambda^{(j)}, U_\lambda^{(j)}$ do not change the spanned vector spaces. 
\end{proposition}
\begin{proof}
See Appendix \ref{sec:proofs}.
\end{proof}

\begin{remark}[Different expressions for basis elements]
The matrices $X_T$, $Y_T$, $X_H$, $Y_H$, $X_\mathbb{R}$, $Y_\mathbb{R}$ could be expressed in several different ways. Using the relationships \eqref{eq:UWPQrelationship}, one obtains some equivalent expressions. For example, if we substitute $J^{-1}P$ for $U$ and $J^{-T}Q$ for $W$, we obtain an equivalent definition of \eqref{eq:complexTXT}, 
\begin{equation}\label{eq:XTequiv1}
    X_T(k, J, P, Q) = J^{-T}QE_k^{a, b}P^T - J^{-1}P E_k^{b, a}Q^T. 
\end{equation}
Moreover, another equivalent expression could also be obtained from the second terms of \eqref{eq:sol}, \eqref{eq:cosol}. For example, applying the map $Z- J^{-T}\!Z^TJ^T$ on $Z = WEP^T\in\cent(\cosq(J))$ and substituting $P$ for $J^TU$ we obtain,
\begin{equation}\label{eq:XTsubstitute}
    X_T(k, J, W, U) = WE_k^{a, b}U^TJ - UE_k^{b, a}W^TJ^T.
\end{equation}
Also setting $Z = UEQ^T$ ($\in\cent(C^{-1})=\cent(C)$) one obtains similar expressions.
\end{remark}

\begin{proposition}\label{thm:XstlambdaXtslambdai}
For a given complex nonsingular $J$ and $C = \cosq(J)$, let $\lambda, \lambda^{-1}$ be eigenvalues of $C$ with the sizes of the Jordan blocks being $r_1, \dots, r_m$. For $j=1,\dots,m$, select the $j^{th}$ Jordan chain matrices $W_\lambda^{(j)}, U_\lambda^{(j)}\in\mathbb{C}^{n\times r_j}$ of $C, C^{-1}$ corresponding to $\lambda$. Similarly select Jordan chains $W_{\frac{1}{\lambda}}^{(j)}, U_{\frac{1}{\lambda}}^{(j)}$. Then, for $1\leq s, t \leq m$, the sets $\bigcup_{j=1}^k\big\{X_T(j, J, W_\lambda^{(s)}, U_\lambda^{(t)})\big\}$ and $\bigcup_{j=1}^k\big\{X_T(j, J, W_\frac{1}{\lambda}^{(t)}, U_\frac{1}{\lambda}^{(s)})\big\}$ span the same vector space. The same results are obtained for the matrices $Y_T, X_H, Y_H$. For a complex eigenvalue $\lambda$ of the real $\cosq(J)$ of a real $J$, the sets $\bigcup_{j=1}^k\big\{X_\mathbb{R}(j, J, W_\lambda^{(s)}, U_\lambda^{(t)})\big\}$ and $\bigcup_{j=1}^k\big\{X_\mathbb{R}(j, J, W_{\overline\lambda}^{(s)}, U_{\overline\lambda}^{(t)})\big\}$ (also true for $Y_\mathbb{R}$) span the same vector space. 
\end{proposition}
\begin{proof}
See Appendix \ref{sec:proofs}. 
\end{proof}

\par Now let us discuss a few examples. 
\begin{example}
Let $J$ and $\cosq(J)$ be
\begin{equation*}
    J = \text{\scalebox{0.8}[0.8]{$\begin{bmatrix} \begin{array}{rrrr}5 & 6 & \mi9 & \mi9 \\
    1 & 0 & \mi1 & 1 \\ \mi3 & \mi6 & 7 & 7 \\ \mi6 & 2 & 2 & 0
    \end{array}\end{bmatrix}$}}, \hspace{0.5cm} \cosq(J)= 
    \frac{1}{2} \text{\scalebox{0.8}[0.8]{$\begin{bmatrix} \begin{array}{rrrr} 0 & 2 & 0 & \mi3 \\ 
    \mi5 & 8 & \mi3 & \mi6 \\ \mi1 & 2 & 1 & \mi3 \\ \mi2 & 0 & 2 & 1 
    \end{array} \end{bmatrix}$}}.
\end{equation*}
The Jordan form of $C = \cosq(J)$ is $J_2^2 \oplus J_2^{1/2}$. Computing generalized eigenvector chains of $C, C^{-1}$ for the eigenvalue $\lambda=2$, we get two Jordan chain matrices
\begin{equation*}
    W_2^{(1)} = \text{\scalebox{0.8}[0.8]{$\begin{bmatrix}\begin{array}{rr}
         6 & 1 \\ 12 & 8 \\ 6 & 1 \\ 0 & 0
    \end{array}\end{bmatrix}$}}, \,\, 
    U_2^{(1)} = \text{\scalebox{0.8}[0.8]{$\begin{bmatrix}\begin{array}{rr}
         \mi4 & 7 \\ \mi8 & 15\\ \mi4 & 8 \\ \mi4 & 7
    \end{array}\end{bmatrix}$}}.
\end{equation*}
From Proposition \ref{thm:XstlambdaXtslambdai} we only need $\sum_{j=1}^2\{X_T(j, J, W_2^{(1)}, U_2^{(1)})\}$,
\begin{gather*}
    W_2^{(1)}E_1^{2,2}(U_2^{(1)})^TJ^T-U_2^{(1)}E_1^{2,2}(W_2^{(1)})^TJ= 
    24\text{\scalebox{0.8}[0.8]{$\begin{bmatrix} 
    \begin{array}{rrrr}5&  \mi1&  \mi3&  0\\
     10&  \mi2&  \mi6&  0\\
      5&  \mi1&  \mi3&  0\\
      4&   0&  \mi4&  0\\
    \end{array}\end{bmatrix}$}},\\
    W_2^{(1)}E_2^{2,2}(U_2^{(1)})^TJ^T-U_2^{(1)}E_2^{2,2}(W_2^{(1)})^TJ = 
    8\text{\scalebox{0.8}[0.8]{$\begin{bmatrix} \begin{array}{rrrr}
    \mi23&  4&  12&   6\\
     \mi46&  5&  30&  12\\
     \mi26&  4&  15&   6\\
     \mi16&  0&  16&   3\\
    \end{array}
    \end{bmatrix}$}}.
\end{gather*}
These two matrices form a basis of the complex vector space $\sol(J)$. 
\end{example}

\begin{example} We discuss the case of a real $J$ (and real solution set $\sol(J)$) where its cosquare has complex eigenvalues. Let $J$ and $C = \cosq(J)$ be given as
\begin{equation*}
    J = \text{\scalebox{0.8}[0.8]{$\begin{bmatrix}\begin{array}{rrrr}
    \mi1 & 0 & \mi3 & \mi2 \\
    1 & 0 & 1 & 0 \\ \mi2 & 2 & 4 & \mi1\\
    0 & \mi1 & \mi1 & \mi2
    \end{array}\end{bmatrix}$}}, \hspace{0.5cm} 
    \cosq(J) = \frac{1}{6} \text{\scalebox{0.8}[0.8]{$\begin{bmatrix}\begin{array}{rrrr}
    6&  \mi2&   \mi6&   1\\
      0&   2&  \mi12&  \mi7\\
      0&   2&    6&   2\\
     \mi6&   4&    6&   4\\
    \end{array}\end{bmatrix}$}}.
\end{equation*}
The four eigenvalues $\lambda, \overline\lambda, 1/\lambda, 1/\overline\lambda$ are $1+i, 1-i, 0.5+0.5i, 0.5-0.5i$. Proceeding as described in Theorem \ref{thm:realsolcosol} and using Proposition \ref{thm:XstlambdaXtslambdai}, we obtain four (real) Jordan chain matrices $W_\lambda^{(1)}, U_\lambda^{(1)}, W_{\frac{1}{\lambda}}^{(1)}, U_{\frac{1}{\lambda}}^{(1)}$:
\begin{equation*}
    \text{\scalebox{0.8}[0.8]{$\begin{bmatrix}\begin{array}{rr}
    2 & 3 \\ 13 & 0 \\ \mi2 & \mi3 \\ \mi4 & \mi6 \end{array}\end{bmatrix}, 
    \begin{bmatrix}\begin{array}{rr}
    \mi5 & 0 \\ \mi4 & 3 \\ \mi1 & \mi3\\ 1 & 3 
    \end{array}\end{bmatrix},
    \begin{bmatrix}\begin{array}{rr}
    \mi5 & 0 \\ \mi4 & 3 \\ \mi1 & \mi3\\ 1 & 3\end{array}\end{bmatrix}, 
    \begin{bmatrix}\begin{array}{rr}
    2 & 3 \\ 13 & 0 \\ \mi2 & \mi3 \\ \mi4 & \mi6 \end{array}\end{bmatrix}$}},
\end{equation*} 
satisfying $CW_{\lambda}^{(1)} = W_{\lambda}^{(1)}\begin{bsmallmatrix}\,1 & 1 \\ \,\mis 1 & 1 \end{bsmallmatrix}$, $C^{-1}U_{\lambda}^{(1)} = U_{\lambda}^{(1)}\begin{bsmallmatrix}\,1 & 1 \\ \,\mis 1 & 1 \end{bsmallmatrix}$ and similar identities for $\frac{1}{\lambda}$. The two basis elements $X_\mathbb{R}(1, J, W_\lambda^{(1)}, U_\lambda^{(1)}), X_\mathbb{R}(1, J, W_{\frac{1}{\lambda}}^{(1)}, U_{\frac{1}{\lambda}}^{(1)})$ are
\begin{equation*}
    \frac{1}{3}\text{\scalebox{0.8}[0.8]{$\begin{bmatrix}\begin{array}{rrrr}
    13&   \mi8&  \mi34&   14\\
       2&  \mi13&  \mi62&  \mi20\\
       8&    8&    7&   10\\
     \mi32&   16&   20&   \mi7 \end{array}\end{bmatrix}$}}, 
    \frac{1}{3}\text{\scalebox{0.8}[0.8]{$\begin{bmatrix}\begin{array}{rrrr}
    \mi39&    9&   22&  \mi32\\
     \mi51&   39&   86&  \mi10\\
      \mi9&   \mi9&  \mi16&    5\\
      36&  \mi18&   10&   16
    \end{array}\end{bmatrix}$}},
\end{equation*}
which form a real basis of the two dimensional linear subspace $\sol(J)$. 
\end{example}

\begin{remark}[Symmetric space]
For a given square matrix $J$ and $C\!=\!\cosq(J)$ let $U$ be the Lie group $\{G : G\text{ invertible}, GC = CG\}$ with its Lie algebra $\mathfrak{u}=\cent(C)$. The decomposition $\mathfrak{u} = \{X:X^T\!J+JX=0\}+\{X:X^T\!J-JX=0\}$ can be obtained by the eigenspaces of the involution $X\mapsto -J^{-1}X^TJ$. This becomes the tangent space of a (not necessarily Riemannian) symmetric space $U/K$ where $K$ is the automorphism group $\{G : G^T\!JG = J\}$. For example, if $J= I_n$ we have a Riemannian symmetric space $\text{GL}(n, \mathbb{R})/\text{O}(n)$. Another example would be $J = I_{p,q}$ where $I_{p,q} = \text{diag}(\underbrace{1, \dots, 1}_{p}, \underbrace{-1, \dots, -1}_{q})$. We obtain the pseudo-Riemannian symmetric space $\text{GL}(n, \mathbb{R}) / \text{O}(p, q)$.% which is the collection of all symmetric matrices with signature $(p, q)$.
\end{remark}
%\alan{make a section somewhere that reminds the reader what is a symmetric space, and proves [or gives some examples]that this one is a non-riemannian symmetric space, }

\subsection{Singular $J$}\label{sec:singularJbackground}

For a nonsingular $J$ the eigenstructure (Jordan form) of the cosquare plays a central role. However, if $J$ is singular the cosquare no longer exists. Rather, we can work with the \textit{Kronecker structure} of the matrix pencil $J - \lambda J^T$. The Kronecker structure of $A-\lambda B$ reveals the usual eigenstructure of $B^{-1}A$ (for invertible $B$), as well as the generalized eigenstructure ($\infty,\frac{0}{0}$ situations) when $B^{-1}A$ is not well defined. In this section $\lambda$ is always the indeterminate. 

\par The structured matrix pencil $J-\lambda J^T$ is called a $T$-palindromic pencil and the Kronecker forms of the palindromic pencils are studied in \cite{schroder2006canonical,schroder2008thesis}. The Kronecker structure of $J-\lambda J^T$ is consisting of (\rn{1}) Jordan blocks of nonzero eigenvalues, (\rn{2}) $(0, \infty)$ Jordan block pairs and (\rn{3}) pairs of singular blocks.

\par Define the set $\mathcal{Z}_J\subset \mathbb{C}^{n\times n}\times\mathbb{C}^{n\times n}$ consisted of pairs of matrices $Z_1, Z_2$
\begin{equation*}
    \mathcal{Z}_J := \{(Z_1, Z_2) : Z_1 (J - \lambda J^T) - (J - \lambda J^T) Z_2 = 0\}.
\end{equation*}
To have $Z_1 (J - \lambda J^T) - (J - \lambda J^T) Z_2 = 0$ for all $\lambda$ it reduces down to two equations
\begin{equation}\label{eq:Z1Z2equation}
    Z_1J = JZ_2 \hspace{0.5cm}\text{ and }\hspace{0.5cm} Z_1J^T = J^TZ_2.
\end{equation}
The set $\mathcal{Z}_J$ is a linear subspace whose dimension can be computed by the outline suggested in \cite{demmel1995dimension}, under the name of the codimension of the orbits. 

\par Recall from Corollary \ref{cor:solcosoldirectsum} that $\sol(J)\,\oplus\,\cosol(J)=\cent(J^{-T}\!J)$. When $J$ is singular, $\sol(J) \,\cap\, \cosol(J)$ is nontrivial. Thus we define the product $\sol(J)\times\cosol(J) = \{(X, Y):X\in\sol(J), Y\in\cosol(J)\}$ which plays the role of $\sol(J)\oplus\cosol(J)$ in previous sections. The following lemma describes what might be called a ``higher order 45 degree rotation" between pairs $(Z_1,Z_2)$ that satisfy equations \eqref{eq:Z1Z2equation} and pairs $(X,Y)$ that consist of solutions and cosolutions:

\begin{lemma}\label{lem:bijectionZandCS}
For a given $J$, $\mathcal{Z}_J$ and $\sol(J)\times\cosol(J)$ are diffeomorphic.
\end{lemma}

\begin{proof}
The map $(Z_1, Z_2)\mapsto (Z_1^T-Z_2, Z_1^T+Z_2)$ is a diffeomorphism from $\mathcal{Z}_J$ to $\sol(J)\times \cosol(J)$ with the inverse map $(X, Y)\mapsto ( (X^T+Y^T)/2, (Y-X)/2 )$. 
\end{proof}

\par As a result, it suffices to compute $\mathcal{Z}_J$ to obtain $\sol(J)$. To see the analogy to Section \ref{sec:nonsingulartheory}, let us assume that $J$ is nonsingular. Explicitly solving \eqref{eq:Z1Z2equation} we obtain $\mathcal{Z}_J = \{(Z^T, J^{-1}Z^TJ) : Z\in\cent(J^{-T}\!J)\}$ with $Z^T = Z_1$ in \eqref{eq:Z1Z2equation}. The map $(Z_1, Z_2)\mapsto Z_1^T-Z_2$ is exactly the map $Z\mapsto Z - J^{-1}Z^TJ$ in Theorem \ref{thm:solandcosol}. 

\par Recall we began with the explicit expression of a basis of $\cent(\cosq(J))$ to compute $\sol(J)$ for a nonsingular $J$. For singular $J$, on the other hand, a basis of $\mathcal{Z}_J$ is less well known but still computable. In particular, we need to compute a pair $(E, F)$ such that $E\cdot K_1(\lambda)^T-K_2(\lambda)\cdot F= 0$ for two Kronecker blocks $K_1$ and $K_2$, which play the role of $E_k^{a, b}$ matrix when $J$ is nonsingular. As discussed in Section 5 of \cite{demmel1995dimension}, (or similarly in \cite{de2011conjtrans,de2011trans}) computing a basis of the collection of all $(E, F)$ can be broken down into computing $(E, F)$ of each component and interaction between components. In Appendix \ref{sec:EFappendix} we give a full basis of the collection of $(E, F)$ pairs for each Kronecker block and interaction. 

\par Furthermore, we define an extension of Jordan chain matrices for matrix pencils. Let $K(\lambda)$ be an $r\times r$ canonical block in the Kronecker structure of an $n\times n$ pencil $J - \lambda J^T$. (Since singular blocks $L_j, L_j^T$ always come in pairs we group them to make a single square canonical block.) Select $n\times r$ matrices $W, U, P, Q$ that satisfy\footnote{These are bases of the \textit{deflating spaces} and in particular for singular pencils they are bases of the \textit{reducing spaces}. Recall that Jordan chains are often thought as a basis of invariant subspaces. Deflating and reducing subspaces \cite{van1983reducing} are extension of the invariant subspace for matrix pencils.} %%% PUT MORE REFENECE
\begin{equation}\label{eq:pencilJordanmatrices}
    (J-\lambda J^T)W = Q \cdot K(\lambda)\hspace{0.5cm}\text{and}\hspace{0.5cm}(J^T-\lambda J)U = P\cdot K(\lambda).
\end{equation}
If $K(\lambda)$ is the usual Jordan block, i.e., $K(\lambda) = J_r^\alpha - \lambda I_r$, \eqref{eq:pencilJordanmatrices} agrees with the definition of the ordinary Jordan chain matrices $W, U, P, Q$ in \eqref{eq:WUPQdefinition}. 

\par The union of all $(QEU^T, WFP^T)$ for $W, U, P, Q$ that satisfy $(J^T-\lambda J)U = P\cdot K_1(\lambda)$, $(J-\lambda J^T)W = Q\cdot K_2(\lambda)$, form a basis of $\mathcal{Z}_J$. Using the mapping in Lemma \ref{lem:bijectionZandCS} on $\mathcal{Z}_J$ and collecting $UE^TQ^T-WFP^T$ one obtains a basis of $\sol(J)$.  

\begin{theorem}\label{thm:complexTsolpencilversion}
Given $J\in\mathbb{C}^{n\times n}$ let $K_1(\lambda), \dots, K_m(\lambda)$ be the (square) Kronecker blocks of $J-\lambda J^T$ with the block sizes $r_1, \dots, r_m$. Select for all $k=1,\dots,m$ the matrices $W^{(k)}, U^{(k)}, P^{(k)}, Q^{(k)}\in\mathbb{C}^{n\times r_k}$ such that $(J-\lambda J^T)W^{(k)} = Q^{(k)} K_k(\lambda)$ and $(J^T-\lambda J)U^{(k)} = P^{(k)} K_k(\lambda)$. For $s\neq t$ let $\{(E_j^{s,t}, F_j^{s,t})\}_{j=1, \dots, 2d}$ be the basis of the collection of $(E, F)$ for the interaction of $K_s(\lambda)$ and $K_t(\lambda)$, listed in Appendix \ref{sec:EFappendix}. Denote the union of all $\{U^{(s)}(E_j^{s, t})^T(Q^{(t)})^T-W^{(t)}F_j^{s,t}(P^{(s)})^T\}_{j=1, \dots, d}$ and $\{U^{(t)}(E_j^{s,t})^T(Q^{(s)})^T - W^{(s)}F_j^{s,t}(P^{(t)})^T\}_{j=d+1, \dots, 2d}$ for all $1\leq s\neq t\leq m$ by $B_{\text{inter}}$. Denote the collection of all matrices $U^{(s)}(E_j^{s, s})^T(Q^{(s)})^T-W^{(s)}F_j^{s,s}(P^{(s)})^T$, for $s = 1, \dots, m$, by $B_{\text{diag}}$. The set $B_{\text{inter}}\cup B_{\text{diag}}$ spans $\sol(J)$. 
\end{theorem}
Alike previous theorems, Theorem \ref{thm:complexTsolpencilversion} can also be extended to $\cosol(J)$ with $U^{(s)}(E_j^{s, t})^T(Q^{(t)})^T+W^{(t)}F_j^{s,t}(P^{(s)})^T$ and $U^{(t)}(E_j^{s,t})^T(Q^{(s)})^T + W^{(s)}F_j^{s,t}(P^{(t)})^T$.

\subsection{The exponential map and the Lie algebra}\label{sec:exp}
In Section \ref{sec:background} we have mostly discussed the tangent space of the group $\{G:G^*\!JG=J\}$. An important tool that connects the Lie group $\{G:G^*\!JG=J\}$ to its tangent space (Lie algebra) is the exponential map, $\exp:\{X:X^*\!J+JX=0\}\to \{G:G^*\!JG=J\}$. 

\par A natural question arises: Will the exponential map recover the whole Lie group $\{G:G^*\!JG=J\}$? The answer to this surjectivity problem is rather complicated, but it has been studied for classical Lie groups, e.g., see \cite{lai1977surjectivity,djokovic1997surjectivity}. To begin with, a classical result states that any connected, compact Lie group $G$ has a surjective exponential map $\exp:\text{Lie}(G)\to G$ \cite{djokovic1997surjectivity}. However this is not enough. 

%We could use this result immediately in some cases. For example, if $J\in\mathbb{R}^{n\times n}$ has its symmetric part $S$ with signature $(n, 0)$ (or $(0, n)$) and nonsingular skew-symmetric part $A$, we have a compact connected $G_J^\mathbb{R}$ which is the intersection of $G_S^\mathbb{R}\cong \text{O}(n)$, $G_A^\mathbb{R}\cong\text{Sp}(2n, \mathbb{R})$, which is again compact and connected. 

\par Nonetheless, there are helpful results that help us understand the exponential map, and, further assist us when using the exponential map numerically. In the real case $G = G_J^\mathbb{R}$, a result by Sibuya \cite{sibuya1960note} is practical: For any matrix $Y\in G_J^\mathbb{R}$, $J\in\mathbb{R}^{n\times n}$ there exists a matrix $X$ in the tangent space, $X\in\{X:X^T\!J+JX=0\}$, such that $\exp(X) = Y$ or $\exp(X) = Y^2$. Furthermore if $Y$ has no real negative eigenvalues we can always find $X$ such that $\exp(X) = Y$. 

% Add quaternion case
% Add few words on square 

\par Numerically, we could take all possible square roots of all $\exp(X)$, $X\in\{X:X^T\!J+JX=0\}$ to obtain the whole group. If the eigenstructure of $\exp(X)$ has $m$ Jordan blocks, there exist $2^m$ matrix square roots (counting non-principal branches). From \cite[Theorem 7.2]{mackey2005structured} it is known that the Jordan blocks of matrices in $G$ always have their counterparts (either $(\lambda, 1/\lambda)$ or $(\lambda, 1/\bar\lambda)$ pair), and one could take non-principal branches of each paired Jordan blocks together to reduce complexity. (So that a matrix square root again has the correct Jordan block pairs.)

\par Chu extends the result of Sibuya to the complex case, as Theorem 7 of \cite{chu1982exponential} states the following: For a complex $J$, the exponential map of $G_J^H$ is surjective, which means every matrix $Y$ in $G_J^H$ could be obtained by exponentiating a matrix $X$ in $\sol^H(J)$. Also for any matrix $Y\in G_J$ we have $X$ in the tangent space $\sol(J)$ such that $\exp(X)=Y$ or $\exp(X) = Y^2$. Again, if $Y$ has no real negative eigenvalue we always have $X$ such that $\exp(X) = Y$ as in the real case.

\section{Dimension count, complex $X^T\!J+JX=0$}\label{sec:complexTdimcount}
In Sections \ref{sec:complexTdimcount} to \ref{sec:realdimcount} we compute the bases of the solution and cosolution by eliminating the overlapping elements. We define few useful direct sums of Kronecker blocks that appear in the Kronecker structure of $J-\lambda J^T$. 

\begin{definition}\label{def:auxmatrixcomplexT}
Define three paired Kronecker blocks (pencils) as follows: 
\begin{equation*}
    \mathcal{L}_n := L_n \oplus L_n^T ,\hspace{1cm}\mathcal{Z}_n := J_n^0 \oplus J_n^\infty, \hspace{1cm}\mathcal{J}_n^\lambda := J_n^\lambda \oplus J_n^{\frac{1}{\lambda}},
\end{equation*}
with sizes $(2n+1)\times (2n+1)$, $2n\times 2n$, and $2n\times 2n$, respectively.
\end{definition}

A simple modification of Theorem 2.1.(a) of \cite{horn2007canonical} is the following lemma. 

\begin{lemma}\label{lem:complexTstructure}
For $J\in\mathbb{C}^{n\times n}$, the Kronecker structure of the pencil $J-\lambda J^T$ could be divided into four parts as follows.
\begin{alignat*}{2}
    K_J = &\,\,\big(\mathcal{L}_{s_1}\oplus \dots \oplus \mathcal{L}_{s_a}\big) \hspace{4cm} &&(\text{singular $L$ block pairs})\\
    & \oplus (\mathcal{Z}_{t_1}\oplus \dots \oplus \mathcal{Z}_{t_b}\big) &&(\text{0, $\infty$ Jordan pairs}) \\
    & \oplus \big(J_{m_1}^1 \oplus \dots \oplus J_{m_c}^1\big) \oplus \big(J_{n_1}^{-1} \oplus \dots \oplus J_{n_d}^{-1}\big) && (\text{$\pm$ 1 Jordan blocks}) \\
    & \oplus \bigoplus_{\substack{j=1 \\ (\lambda_j\neq \pm1, 0, \infty) }}^e \mathcal{J}_{p_j}^{\lambda_j} && (\text{$\lambda_j, \lambda_j^{-1}$ Jordan pairs}).
\end{alignat*}
In the following theorem, we will say $\lambda_j \sim \lambda_k$ when $\{\lambda_j, \lambda_j^{-1}\}=\{\lambda_k, \lambda_k^{-1}\}$, i.e., when $\mathcal{J}_{p_j}^{\lambda_j}$ and $\mathcal{J}_{p_k}^{\lambda_k}$ represent the same Kronecker structure. 
\end{lemma}

\begin{boxtheo}\label{thm:complexTsoldim}
Let $J$ be an $n\times n$ complex matrix with the Kronecker structure $K_J$ of the pencil $J-\lambda J^T$ given as Lemma \ref{lem:complexTstructure}. Then, the complex dimension of the solution of $X^T\!J+JX=0$ is the sum of: 
\begin{enumerate}[label=(\alph*)]
    \item Dimension $D_L$ from the singular blocks $\mathcal{L}_{s_j}$
    \begin{equation*}
        D_L = \sum_{j=1}^a (s_j+1) + \sum_{j < k} \max(2s_j+1, 2s_k+1) + (\#\text{ of }s_j=s_k).
    \end{equation*}
    \item Dimension $D_Z$ from $0$ and $\infty$ block pairs $\mathcal{Z}_{t_j}$
    \begin{equation*}
        D_Z = \sum_{j=1}^b t_j + \sum_{j<k} \min(2t_j, 2t_k).
    \end{equation*}
    \item Dimensions $D_{1}$ and $D_{-1}$ from $\pm 1$ Jordan blocks
    \begin{equation*}
        D_1 = \sum_{j=1}^c \Big\lfloor \frac{m_j}{2}\Big\rfloor + \sum_{j<k} \min(m_j, m_k), \hspace{0.3cm} D_{-1} = \sum_{j=1}^d \Big\lceil \frac{n_j}{2}\Big\rceil + \sum_{j<k} \min(n_j, n_k).
    \end{equation*}
    \item Dimension $D_P$ from all other paired Jordan blocks $\mathcal{J}_{p_j}^{\lambda_j}$
    \begin{equation*}
        D_P = \sum_{j=1}^e p_j + \sum_{\lambda_j\sim\lambda_k} \min(2p_j, 2p_k).
    \end{equation*}
    \item Dimension $D_I$ from the interaction of $\mathcal{L}$ blocks and the others
    \begin{equation*}
        D_I = \sum_{j=1}^a \Big( n - \sum_{k=1}^a (2s_k+1) \Big) = a (n - \sum_{k=1}^a (2s_k+1)).
    \end{equation*}
\end{enumerate}
\end{boxtheo}
\begin{proof}
\textit{(d) Blocks $\mathcal{J}_{p_j}^{\lambda_j}$ with eigenvalue pairs $\{\lambda_j, 1/\lambda_j\}$}: We begin with the most generic case. Recall that the blocks of the centralizer do not interact if they have different eigenvalues \cite{arnold1971matrices}. Since the solutions are projected from the centralizer of the cosquare of the nonsingular part of $J$, we also have that the solutions of different eigenvalues do not interact. Thus we focus on a fixed $\lambda$. Using the same setting as Theorem \ref{thm:complexTsolcosol}, abbreviate the matrix $X_T(k, J, W_\lambda^{(s)}, U_\lambda^{(t)})$ by $X_{\lambda,k,s,t}$ (and $Y_T$ by $Y_{\lambda,k,s,t}$). By Proposition \ref{thm:XstlambdaXtslambdai} we can eliminate linearly dependent elements in $B_T^+$ of $\eqref{eq:complexTsol}$, reducing $\bigcup_{k,s,t}\{X_{\lambda,k,s,t},X_{\frac{1}{\lambda},k,s,t}\}$ to $\bigcup_{k,s,t}\{X_{\lambda,k,s,t}\}$. Similarly for $B_T^-$ of \eqref{eq:complexTcosol} we reduce the set $\bigcup_{k,s,t}\{Y_{\lambda,k,s,t},Y_{\frac{1}{\lambda},k,s,t}\}$ to $\bigcup_{k,s,t}\{Y_{\lambda,k,s,t}\}$. So far we have the maximum sum of the solution and cosolution dimensions equal to the dimension of the subset of the centralizer corresponding to the $\lambda, 1/\lambda$ Jordan structures of the cosquare. Thus we conclude $\{X_{\lambda,k,s,t}:\forall k, s, t\}$ and $\{Y_{\lambda,k,s,t}:\forall k, s, t\}$ are both already linearly independent basis sets. The dimension $D_P$ follows as
\begin{equation*}
    D_P = \underbrace{\sum_{j=1}^e p_j}_{\text{All $X_{\lambda,k,j,j}$}} + \underbrace{\sum_{\lambda_j=\lambda_k}\min(2p_j, 2p_k)}_{\text{$X_{\lambda,k,j,k}$ with $j\neq k$}}.
\end{equation*}

\par \textit{(c) Jordan blocks with $\pm 1$ eigenvalues}: The Jordan blocks with $\lambda=\pm 1$ have a special property that $W_{\frac{1}{\lambda}}$ and $U_{\frac{1}{\lambda}}$ coincide with $W_\lambda$ and $U_\lambda$ respectively for all Jordan chain matrices. By Proposition \ref{thm:XstlambdaXtslambdai} the sets $\bigcup_k\{X_T(k,J,W_\lambda^{(s)},U_\lambda^{(t)})\}$ and $\bigcup_k\{X_T(k,J,W_\lambda^{(t)},U_\lambda^{(s)})\}$ span the same vector space (similarly for $Y_T$) if $s\neq t$. Using the same logic in the proof of (d), the maximum dimension sum of the two sets is the dimension of the $(s,t), (t,s)$ interaction of the centralizer. Thus we obtain the off-diagonal basis $\bigcup_{s<t}\bigcup_k\{X_T(k,J,W_\lambda^{(s)},U_\lambda^{(t)})\}$. We are left to determine the basis in the set $\bigcup_k\{X_T(k,J,W_\lambda^{(s)},U_\lambda^{(s)}\}$. Let us first consider $\lambda=1$. Since $J^TU_1$ = $JU_1$ we have $X_T$ equal to (let us drop the superscript $^{(s)}$ for a moment) $W_1 E_k U_1^TJ - U_1 E_k^T W_1^TJ$. Using the $K = K_s$ in the proof of Proposition \ref{thm:XstlambdaXtslambdai} there exist $U_1', W_1'$ such that $U_1'K=W_1$ and $W_1'K=U_1$. Moreover from the proof of Proposition \ref{thm:choiceofWU} we have $T_u, T_w$ such that $U_1' = U_1T_u$ and $W_1' = W_1T_w$ holds. Observe that $T_uKE_kK^TT_w^T=\sum_{j=1}^k c_jE_k$ with $c_k=(-1)^{k-1}$. If $k=1$ we have 
\begin{align*}
    X_T(1,J,W_1,U_1) &= W_1E_1 U_1^TJ - U_1 E_1^T W_1^TJ \\
    &= U_1T_uKE_1K^TT_w^TW_1^TJ-W_1T_wKE_1^TK^TT_u^TU_1^TJ,
\end{align*}
which becomes $X_T(1,J,W_1,U_1) = -X_T(1,J,W_1,U_1)$ and thus equals zero. Similarly for an odd $k$ $X_T(k,J,W_1,U_1)$ and $\bigcup_{j=1}^{k-1}\{X_T(j,J,W_1,U_1)\}$ are linearly dependent and for an even $k$ $Y_T(k,J,W_1,U_1)$ and $\bigcup_{j=1}^{k-1}\{Y_T(j,J,W_1,U_1)\}$ are linearly dependent. Since the dimension of $ \text{span}(\{X_T(k,J,W_1,U_1), Y_T(k,J,W_1,U_1)\})$ is $1$ we deduce that the dimensions of the solution and cosolution are $\lfloor m_s/2 \rfloor$ and $\lceil m_s/2 \rceil$, respectively. For $\lambda=-1$ it is similar except that $-J^TU_\lambda$ = $JU_\lambda$, which leads to the opposite; For $\lambda=-1$, $X_T$ does not add dimension if $k$ is even and $Y_T$ does not add dimension if $k$ is odd. Combining these two results we have
\begin{equation*}
    D_1 \hspace{-0.1cm}= \hspace{-0.2cm}\underbrace{\sum_{j=1}^c \lfloor \frac{m_j}{2}\rfloor}_{\text{From $X_{1,k,j,j}$}}\!\!+ \underbrace{\sum_{\lambda_j=\lambda_k}\!\!\min(m_j, m_k)}_{\text{$X_{1,k,j,k}$ with $j\neq k$}} \hspace{0.1cm}\text{ and }\hspace{0.1cm} D_{-1}\hspace{-0.1cm} = \hspace{-0.2cm}\underbrace{\sum_{j=1}^c \lceil \frac{m_j}{2}\rceil}_{\text{From $X_{-1,k,j,j}$}} \!\!+ \underbrace{\sum_{\lambda_j=\lambda_k}\!\!\min(m_j, m_k)}_{\text{$X_{-1,k,j,k}$ with $j\neq k$}}.
\end{equation*}

\par \textit{(b) $0$ and $\infty$ Jordan pair blocks $\mathcal{Z}_{t_j}$}: The $0, \infty$ Jordan pair $\mathcal{Z}_{t_j} = J_{t_j}^0 \oplus J_{t_j}^\infty$ in $K_J$ corresponds to the even sized $2t_j\times 2t_j$ Jordan block of eigenvalue $0$ in terms of the congruence canonical form. Although they contain zero and $\infty$ eigenvalues, we can treat them as a pair of eigenvalues $\lambda, 1/\lambda$ and proceed as above. The dimension count and the linearly independent basis elements are identical to the situation of (d). 

\par \textit{(a) Left-right singular pairs $\mathcal{L}_{s_j}$}: The Kronecker block $\mathcal{L}_{s_j}$ corresponds to the congruence canonical matrix $J_{2s_j+1}^0$ and this cannot be made into a nonsingular block by the M\"{o}bius transformation. In this case we use Theorem \ref{thm:complexTsolpencilversion} and the given basis of all $(E, F)$ blocks in Appendix \ref{sec:EFappendix}. Fix $\mathcal{L}_{s_j}$ and denote the canonical blocks provided in Appendix \ref{sec:EFsingular1} by $(E_1, F_1), \dots, (E_{2s_j+2}, F_{2s_j+2})$. We outline a similar technique to the proof of (d) to deduce that $UE_k^TQ^T - WF_kP^T$ and $UE_{k+s_j+1}^TQ^T - W F_{k+s_j+1} P^T$ are linearly dependent. (We drop the superscript $^{(j)}$ of $W, U, P, Q$.) Since we have a $T$-palindromic Kronecker structure \cite{schroder2006canonical} the canonical block $\mathcal{S}_{2s_j+1}$ (defined in Appendix \ref{sec:EFappendix}) is used instead of $\mathcal{L}_{s_j}$. From the construction $A^T(J-\lambda J^T)A = \mathcal{S}_{2s_j+1}\oplus \cdots$ of $\mathcal{S}_{2s_j+1}$ block, the first $2s_j+1$ columns of $A$ and $A^{-T}$ eligilble for $W$ and $Q$, as defined in \eqref{eq:pencilJordanmatrices}. Observe that $U$ can be selected as the matrix with the permuted columns of $W$ in the order of $s_j+1, s_j, \dots, 1, 2s_j+1, 2s_j, \dots, s_j+1$. (Select $P$ similarly with the same permuted columns of $Q$.) Then from the shapes of $(E_k, F_k)$ and $(E_{2s_j+2-k}, F_{2s_j+2-k})$ we find $UE_k^TQ^T=WF_{2s_j+2-k}P^T$ and $WF_kP^T = UE_{2s_j+2-k}^TQ^T$. The dimensions of $\sol(J)$ and $\cosol(J)$ are at most $s_j+1$ each and the maximum dimension sum is $2s_j+2$, obtaining a basis $\bigcup_{k=1}^{s_j+1}\{UE_k^TQ^T-PF_kW^T\}$ of $\sol(J)$ corresponding to $\mathcal{L}_{s_j}$. For the interaction between $\mathcal{L}_{s_j}$ and $\mathcal{L}_{s_k}$, (assume $s_j\leq s_k$ and let the dimension of $(E, F)$ from appendices \ref{sec:EFsingular1} or \ref{sec:EFsingular2} be $2d$) the set $\bigcup_{k=1}^{d}\{UE_k^TQ^T-PF_kW^T\}$ and $\bigcup_{k=d+1}^{2d}\{UE_k^TQ^T-PF_k^TW^T\}$ are linearly dependent as in Proposition \ref{thm:XstlambdaXtslambdai}. 

\par \textit{(e) Interaction between $\mathcal{L}_{s_j}$ and all other blocks}: As we discuss in Appendix \ref{sec:EFappendix}, only the interactions between the singular pairs $\mathcal{L}_{s_j}$ and the Jordan blocks are left. Let us fix a singular pair $\mathcal{L}_{s_j}$ and a Jordan block $J_{p_k}^{\lambda_k}$. From Appendix \ref{sec:EFsingularandJordan} the dimension of the collection of all $(E, F)$ is $2p_k$ and it is independent of $s_j$. Since $(E_s, F_s) = (F_{s+p_k, t+p_k}^T, E_{s+p_k, t+p_k}^T)$ from the derivation, it is easily verified that only the first $p_k$ elements of the given basis of the collection of $(E, F)$ create independent elements to $UE_s^TQ^T-PF_sW^T$. Thus we have the dimension of the solution and cosolution both equal to $p_k$. Adding the dimension for all possible choice of Jordan block we have the dimension of the interaction $(\text{Size of all Jordan blocks}) = n -(\text{size of all } \mathcal{L} \text{ blocks}) = n - \sum(2s_l+1)$. Summing for all $j$ we obtain $D_I$.
\end{proof}

In the course of the proof of Theorem \ref{thm:complexTsoldim} we have also identified the linearly dependent elements in $B_T^-(J)$. A reader can obtain a basis of $\cosol(J)$ in the same manner. We briefly state the dimension of the cosolution set as follows.
\begin{corollary}
For a given complex $J$ with the Kronecker structure of $J-\lambda J^T$ in Lemma \ref{lem:complexTstructure}, the complex dimension of the solution of $X^T\!J-JX=0$ is
\begin{equation*}
    D_L + D_Z + D_P + D_I + \underbrace{\sum_{j=1}^c\Big\lceil \frac{m_j}{2}\Big\rceil + \sum_{j<k}\min(m_j, m_k) + \sum_{j=1}^d \Big\lfloor \frac{n_j}{2}\Big\rfloor+ \sum_{j<k}\min(n_j, n_k)}_{\dim(\cent(\pm 1\text{ blocks})) - D_1 - D_{-1}},
\end{equation*}
where $D_L, D_Z, D_P, D_I$ are identical to the ones defined in Theorem \ref{thm:complexTsoldim}.
\end{corollary}

\section{Dimension count for complex $X^H\!J+JX=0$}\label{sec:complexHdimcount}
Recall that $\dim(\sol^H(J))$ is equal to $\dim(\cosol^H(J))$ since $\sol(J) = i\cosol(J)$. The dimension of $\sol^H(J)$ (and also $\cosol^H(J)$) is just half of the codimension of the orbit of $J-\lambda J^H$, which could be computed by \cite[Theorem 2.2]{demmel1995dimension}.

On the other hand, the following Theorem \ref{thm:complexHsoldim} computes the dimension (which gives the same result as above) by providing the precise basis set of $\sol^H(J)$. The Kronecker structure of the pencil $J-\lambda J^H$ always has $(\lambda, 1/\bar\lambda)$ pair \cite{schroder2006canonical}. We begin by defining a paired Kronecker block as in Definition \ref{def:auxmatrixcomplexT}. 

\begin{definition}
Define the paired $2n\times 2n$ Kronecker block $J_n^{\lambda, *}$ for a pair $\lambda$, $1/{\bar\lambda}$
\begin{equation*}
    \mathcal{J}_n^{\lambda, *} := J_n^\lambda \oplus J_n^{1/{\bar{\lambda}}} . 
\end{equation*}
\end{definition}

Again modifying Theorem 2.1.(c) of \cite{horn2007canonical}, we obtain the following lemma. 

\begin{lemma}\label{lem:complexHstructure}
The Kronecker structure of the pencil $J-\lambda J^H$ for $J\in\mathbb{C}^{n\times n}$ is
\begin{alignat*}{2}
    K_J = &\,\,\big(\mathcal{L}_{s_1}\oplus \dots \oplus \mathcal{L}_{s_a}\big) \hspace{2.5cm} &&(\text{singular $L$ block pairs})\\
    & \oplus (\mathcal{Z}_{t_1}\oplus \dots \oplus \mathcal{Z}_{t_b}\big) &&(\text{0 and $\infty$ Jordan pairs}) \\
    & \oplus \big(J_{m_1}^{\alpha_1} \oplus \dots \oplus J_{m_c}^{\alpha_c}\big)  && (\text{$|\alpha_j|=1$ Jordan blocks}) \\
    & \oplus \bigoplus_{\substack{j=1 \\ (|\lambda_j|\neq 0, 1, \infty)}}^d \mathcal{J}_{p_j}^{\lambda_j, *} && (\text{$\lambda_j, 1/\bar{\lambda_j}$ Jordan pairs}).
\end{alignat*}
In the following theorem, we will say $\lambda_j\sim \lambda_k$ when $\{\lambda_j, 1/\bar{\lambda}_j\}=\{\lambda_k, 1/\bar{\lambda}_k\}$, i.e., when $\mathcal{J}_{p_j}^{\lambda_j, *}$ and $\mathcal{J}_{p_k}^{\lambda_k, *}$ represent the same Kronecker structure.
\end{lemma}

\begin{boxtheo}\label{thm:complexHsoldim}
Let $J$ be an $n\times n$ complex matrix with the Kronecker structure $K_J$ of the pencil $J-\lambda J^H$ given as Lemma \ref{lem:complexHstructure}. Then, the real dimension of the solution set of $X^H\!J+JX=0$ is the sum of:
\begin{enumerate}[label=(\alph*)]
    \item Dimension $D_L$ from the singular blocks $\mathcal{L}_{s_j}$
    \begin{equation*}
        D_L = \sum_{j=1}^a (2s_j+2) + \sum_{j < k} 2\max(2s_j+1, 2s_k+1) + 2(\#\text{ of }s_j=s_k).
    \end{equation*}
    \item Dimension $D_Z$ from $0$ and $\infty$ block pairs $\mathcal{Z}_{t_j}$
    \begin{equation*}
        D_Z = \sum_{j=1}^b 2t_j + \sum_{j<k} \min(4t_j, 4t_k).
    \end{equation*}
    \item Dimension $D_\alpha$ Jordan blocks $J_{m_j}^{\alpha_j}$ with eigenvalues $|\alpha_j|=1$
    \begin{equation*}
        D_\alpha = \sum_{j=1}^c m_j + \sum_{\alpha_j=\alpha_k} \min(2m_j, 2m_k).
    \end{equation*}
    \item Dimension $D_P$ from all other paired Jordan blocks $\mathcal{J}_{p_j}^{\lambda_j}$
    \begin{equation*}
        D_P = \sum_{j=1}^d 2p_j + \sum_{\lambda_j\sim\lambda_k}\min(4p_j, 4p_k).
    \end{equation*}
    \item Dimension $D_I$ from the interaction of $\mathcal{L}$ blocks and the others
    \begin{equation*}
        D_I = \sum_{j=1}^a \Big(2n - \sum_{k=1}^a (4s_k+2) \Big) = a (2n - \sum_{k=1}^a(4s_k+2)).
    \end{equation*}
\end{enumerate}
\end{boxtheo}
\begin{proof}
The basis elements corresponding to (a), (b), (d), and (e) could be determined similarly as in the proof of Theorem \ref{thm:complexTsoldim}, except that in this case we are computing the real dimensions. The real dimensions are $2$ times the complex dimensions in Theorem \ref{thm:complexTsolcosol}. 
\par \textit{(c) Jordan blocks with eigenvalues $|\alpha_j|=1$}: For eigenvalues $\alpha_j$ such that $|\alpha_j|=1$ we have $\alpha_j = 1/\overline{\alpha}_j$ which means we are in a same situation as in (c) of Theorem \ref{thm:complexTsoldim}. Using the same technique it can be proved that for each $k$, $iY_H(k,J,W_\lambda, U_{\bar\lambda})$ is linearly dependent to $\bigcup_{j=1}^k X_H(k, J, W_\lambda, U_{\bar\lambda})$ and $iX_H(k,J,W_\lambda, U_{\bar\lambda})$ is linearly dependent to $\bigcup_{j=1}^k Y_H(k, J, W_\lambda, U_{\bar\lambda})$. With a similar argument used in the proof of Theorem \ref{thm:complexTsoldim}, by examining the maximum sum of the dimensions, we deduce that the sets $\bigcup_{j=1}^k X_H(k, J, W_\lambda, U_{\bar\lambda})$ and $\bigcup_{j=1}^k Y_H(k, J, W_\lambda, U_{\bar\lambda})$ for all Jordan blocks (and their interactions) are linearly independent bases of the solution set and the cosolution set. 
\end{proof}
\begin{corollary}
For a given complex $J$ with the Kronecker structure of $J-\lambda J^H$ given as in Theorem \ref{thm:complexHsoldim}, the real dimension of the solution to $X^H\!J-JX=0$, i.e., $\dim(\cosol^H(J))$, is the same as $\dim(\sol^H(J))$ described in Theorem \ref{thm:complexHsoldim}.
\end{corollary}

\section{Dimension count for real $X^T\!J+JX=0$}\label{sec:realdimcount}
The dimension and a basis of the real solution of $X^T\!J+JX=0$ when $J$ is real is determined. We define another canonical block that appears in the Kronecker structure of real $J-\lambda J^T$ as follows. 
\begin{definition}
For $\lambda \in \mathbb{C}\backslash\mathbb{R}$, define a block diagonal matrix,
\begin{equation}
    \mathcal{J}_n^{\lambda, \mathbb{R}} := \realify{J_n^\lambda} \oplus \realify{J_n^{1/\lambda}} \hspace{0.5cm}\in\mathbb{R}^{4n\times 4n}.
\end{equation}
The matrix $\mathcal{J}_n^{\lambda, \mathbb{R}}$ represents the four Jordan blocks $J_n^\lambda$, $J_n^{\bar\lambda}$, $J_n^{1/\lambda}$, $J_n^{1/\bar{\lambda}}$ of the Kronecker structure at once.  
\end{definition}
From Theorem 2.1.(d) of \cite{horn2007canonical} we obtain the following lemma for the real case. 
\begin{lemma}\label{lem:realstructure}
Let $\Omega$ be the union of $\mathbb{R}$ and the complex unit circle minus the points $\{-1, 0, 1\}$. The Kronecker structure of the real pencil $J-\lambda J^T$ for $J\in\mathbb{R}^{n\times n}$ is
\begin{alignat*}{2}
    K_J &= \,\,\big(\mathcal{L}_{s_1}\oplus \dots \oplus \mathcal{L}_{s_a}\big) \hspace{3.5cm} &&(\text{singular $L$ block pairs})\\
    & \oplus (\mathcal{Z}_{t_1}\oplus \dots \oplus \mathcal{Z}_{t_b}\big) &&(\text{0 and $\infty$ Jordan pairs}) \\
    & \oplus \big(J_{m_1}^1 \oplus \dots \oplus J_{m_c}^1\big) \oplus \big(J_{n_1}^{-1} \oplus \dots \oplus J_{n_d}^{-1}\big) && (\text{$\pm$ 1 Jordan blocks}) \\
    & \oplus \big(\mathcal{J}_{p_1}^{\alpha_1}\oplus\dots\oplus\mathcal{J}_{p_e}^{\alpha_e}\big) && (\text{$\alpha_j,\alpha_j^{-1}\in\Omega$ Jordan pairs}) \\
    & \oplus \bigoplus_{\substack{j=1 \\ (\lambda_j\in\mathbb{C}\backslash\mathbb{R})}}^f \mathcal{J}_{q_j}^{\lambda_j, \mathbb{R}} && (\text{$\lambda_j$, $\bar\lambda_j$, $\lambda_j^{-1}$, ${\bar\lambda_j}^{-1}$ blocks}).
\end{alignat*}
In the following theorem, we will say $\alpha_j\sim\alpha_j$ (resp. $\lambda_j\sim\lambda_k$) when $\{\alpha_j, \frac{1}{\alpha_j}\}=\{\alpha_k, \frac{1}{\alpha_k}\}$ (resp. $\{\lambda_j, \bar\lambda_j, \frac{1}{\lambda_j}, \frac{1}{\overline\lambda_j}\}=\{\lambda_k, \bar\lambda_k, \frac{1}{\lambda_k}, \frac{1}{\overline\lambda_k}\}$), i.e., when $J_{p_j}^{\alpha_j}$ and $J_{p_k}^{\alpha_k}$ (resp. $\mathcal{J}_{q_j}^{\lambda_j, \mathbb{R}}$ and $\mathcal{J}_{q_j}^{\lambda_j, \mathbb{R}}$ ) represent the same Kronecker structure. 
\end{lemma}

\begin{boxtheo}\label{thm:realsoldim}
Let $J$ be an $n\times n$ real matrix with Kronecker structure $K_J$ of the pencil $J-\lambda J^T$ given as Lemma \ref{lem:realstructure}. Then, the real dimension of the real solution set of $X^T\!J+JX=0$ is the sum of:
\begin{enumerate}[label=(\alph*)]
    \item Dimension $D_L$ from the singular blocks $\mathcal{L}_{s_j}$
    \begin{equation*}
        D_L = \sum_{j=1}^a (s_j+1) + \sum_{j < k} \max(2s_j+1, 2s_k+1) +(\#\text{ of }s_j=s_k).
    \end{equation*}
    \item Dimension $D_Z$ from $0$ and $\infty$ block pairs $\mathcal{Z}_{t_j}$
    \begin{equation*}
        D_Z = \sum_{j=1}^b t_j + \sum_{j<k} \min(2t_j, 2t_k).
    \end{equation*}
    \item Dimensions $D_{1}$ and $D_{-1}$ from $\pm 1$ Jordan blocks
    \begin{equation*}
        D_1 = \sum_{j=1}^c \Big\lfloor \frac{m_j}{2}\Big\rfloor + \sum_{j<k} \min(m_j, m_k), \hspace{0.3cm} D_{-1} = \sum_{j=1}^d \Big\lceil \frac{n_j}{2}\Big\rceil + \sum_{j<k} \min(n_j, n_k).
    \end{equation*}
    \item Dimension $D_\alpha$ from Jordan block pairs $\mathcal{J}_{p_j}^{\alpha_j}$
    \begin{equation*}
        D_\alpha = \sum_{j=1}^e p_j + \sum_{\alpha_j\sim\alpha_k} \min(2p_j, 2p_k).
    \end{equation*}
    \item Dimension $D_P$ from blocks $\mathcal{J}_{q_j}^{\lambda_j, \mathbb{R}}$ with all other $\lambda_j$, $\bar\lambda_j$, $\lambda_j^{-1}$, ${\bar\lambda_j}^{-1}$
    \begin{equation*}
        D_P = \sum_{j=1}^f 2q_j + \sum_{\lambda_j\sim\lambda_k}\min(4q_j, 4q_k).
    \end{equation*}
    \item Dimension $D_I$ from the interaction of $\mathcal{L}$ blocks and the rest
    \begin{equation*}
        D_I = a (n - \sum_{i=1}^a s_i ).
    \end{equation*}
\end{enumerate}
\end{boxtheo}
\begin{proof}
The basis elements corresponding to (a), (b), (c) and (f) are computed identically as in the proof of Theorem \ref{thm:complexTsoldim}. (The proofs of (a), (b), (c), (f) in Theorem \ref{thm:complexTsoldim} does not assume a complex $J$.) 
\par \textit{(d) Jordan block pairs $\mathcal{J}_{p_j}^{\alpha_j}$ with $\alpha_j\in\Omega$}: If $\alpha_j\in\mathbb{R}$, the corresponding basis elements are determined as in (d) of Theorem \ref{thm:complexTsoldim}. If $\alpha_j$ is on the unit circle, the basis elements are determined by applying Proposition \ref{thm:XstlambdaXtslambdai} (deleting $X_\mathbb{R}$ and $Y_\mathbb{R}$ matrices corresponding to $\frac{1}{\alpha_j}$) and using the maximum dimension argument. 
\par \textit{(e) Jordan blocks of $\lambda_j$, $\bar\lambda_j$, $\lambda_j^{-1}$, ${\bar\lambda_j}^{-1}$}: Proposition \ref{thm:XstlambdaXtslambdai} is used to delete the elements corresponding to the eigenvalues $\bar\lambda_j$ and $\bar\lambda_j^{-1}$. 
\end{proof}

\section{Computing the solutions numerically and generating plots}\label{sec:numerical}

In this section we discuss numerical applications related to the group $\{G:G^*\!JG = J\}$ and its tangent space $\{X:X^*\!J+JX = 0\}$. 

\subsection{Sampling random matrices from $\{G:G^*\!JG = J\}$ for a given $J$}
Although the main focus of Sections \ref{sec:complexTdimcount} to \ref{sec:realdimcount} is the tangent space $\sol(J)$ (and $\sol^H(J)$) of the group $G_J$, the computed basis of $\sol(J)$ could be used to sample and plot the identity component of $G_J$ by the exponential map. The surjectivity of the exponential map is addressed in Section \ref{sec:exp}. The following simple algorithm is one way to sample $N$ random elements from the identity component of the group $\{G:G^*\!JG = J\}$. 

\begin{algorithm}
\caption{Sampling $N$ random points from $\{G:G^*\!JG = J\}$}
\begin{algorithmic}
\Require Square matrix $J$, number of samples $N$
\State $\mathcal{G}\gets$ Empty vector of $N$ matrices
\State $\mathcal{S} \gets$ Vector of basis elements of $\sol(J)$ obtained from Theorems \ref{thm:complexTsoldim}, \ref{thm:complexHsoldim}, or \ref{thm:realsoldim}
\State $m\gets \text{size}(\mathcal{S})$
\For{$i \leq N$} 
    \State $r \gets $ Length $m$ random vector
    \State $\mathcal{G}[i] \gets \exp(\sum_j r[j]\cdot \mathcal{S}[j])$
\EndFor
\Ensure $\mathcal{G}$
\end{algorithmic}
\label{alg:randompoints}
\end{algorithm}

\par In the case of $G_J^H$, Algorithm \ref{alg:randompoints} samples the whole Lie group, and for $G_J$ and $G_J^\mathbb{R}$ it samples all elements except for the matrices that have real negative eigenvalues, as discussed in Section \ref{sec:exp}.

\par A possible application of Algorithm \ref{alg:randompoints} is sampling test points for structured matrix computations \cite{benner1998numerically,bunse1992chart,fassbender2007symplectic}. 
%potential algorithms, include testings on structured matrix computations

\subsection{Plotting 3D projections of the group $\{G:G^*\!JG = J\}$}

Using Algorithm \ref{alg:randompoints} one can sample points of $\{G:G^*\!JG = J\}$ but generally they lie inside higher dimensional manifolds which cannot be visualized directly. Algorithm \ref{alg:3dprojection} is one obvious way to visualize such a manifold (in particular the group $G_J^\mathbb{R}$) using a random three dimensional projection. Plotting functions such as \verb|scatter| are useful for creating the projected images. 

\begin{algorithm}
\caption{Plotting a 3D projection of $G_J^\mathbb{R}$}
\begin{algorithmic}
\Require $n\times n$ real matrix $J$, number of samples $N$
\State $\mathcal{G}\gets N\text{ random points of }\{G:G^T\!JG=J\}$ from Algorithm \ref{alg:randompoints}
\State $X \gets $Empty $N\times 3$ matrix
\State $Q\gets n^2\times 3$ ``tall skinny" orthogonal matrix 
\For{$i \leq N$} 
    \State $X[i,:]\gets \verb|vec|(\mathcal{G}[i])^T\cdot Q$
\EndFor
\State $\verb|scatter|(X[1,:], X[2,:], X[3,:])$
\end{algorithmic}
\label{alg:3dprojection}
\end{algorithm}

Figure \ref{fig:rand3Dscatter} provides some examples of randomly projected three dimensional scatter plots. Figure \ref{fig:rand3Dscatter} is two sets of 50000 randomly sampled points of the group $G_J^\mathbb{R}$ for $J\in\mathbb{R}^{8\times 8}$, scattered in $\mathbb{R}^3$ using the programming language \verb|Julia|. 

\begin{figure}[ht]
    \centering
    \includegraphics[width=4.9in]{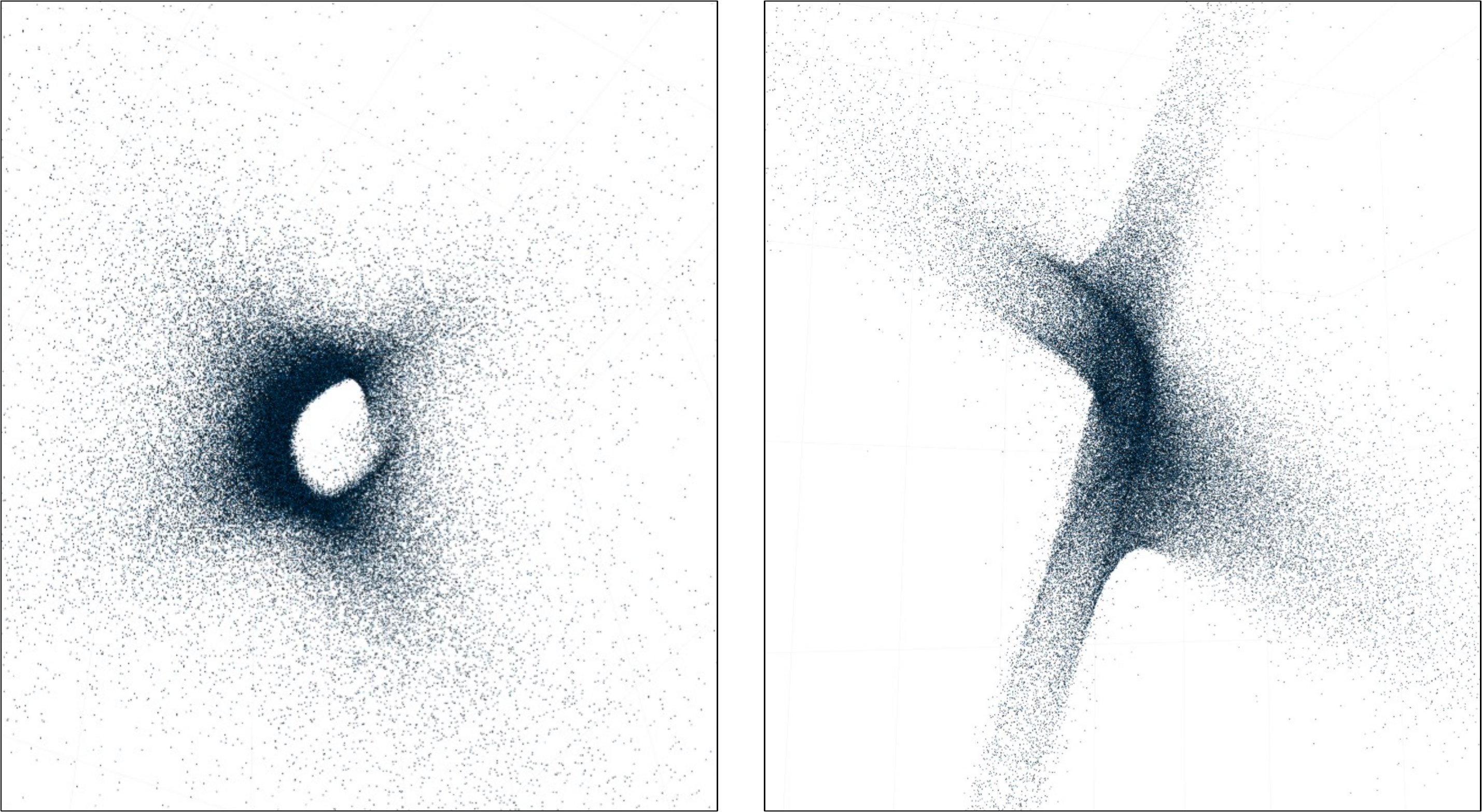}
    \caption{Scatter plots of the automorphism groups of two randomly selected $J\in\mathbb{R}^{8\times 8}$, created with the programming language Julia \cite{bezanson2017julia}. Each plot is 50000 sample points from the four dimensional manifold $G_J^\mathbb{R}$ projected onto the column space of a randomly selected $Q$ in Algorithm \ref{alg:3dprojection} and scattered in $\mathbb{R}^3$.}
    \label{fig:rand3Dscatter}
\end{figure}

\par Additionally if the group is a two dimensional surface so that the basis set $\mathcal{S}$ has two matrices, one can use the plotting function \verb|surface| instead of \verb|scatter|, by substituting $r[1], r[2]$ values in Algorithm \ref{alg:randompoints} by grid vertices. Examples of the visualization created from the plotting function \verb|surface| are used in Figure \ref{fig:4x4plots} (using the package \verb|Makie.jl| \cite{DanischKrumbiegel2021}) where the plotting function \verb|scatter| creates visualizations as the following Figure \ref{fig:rand2Dscatter}.

\begin{figure}[ht]
    \centering
    \includegraphics[width=4.9in]{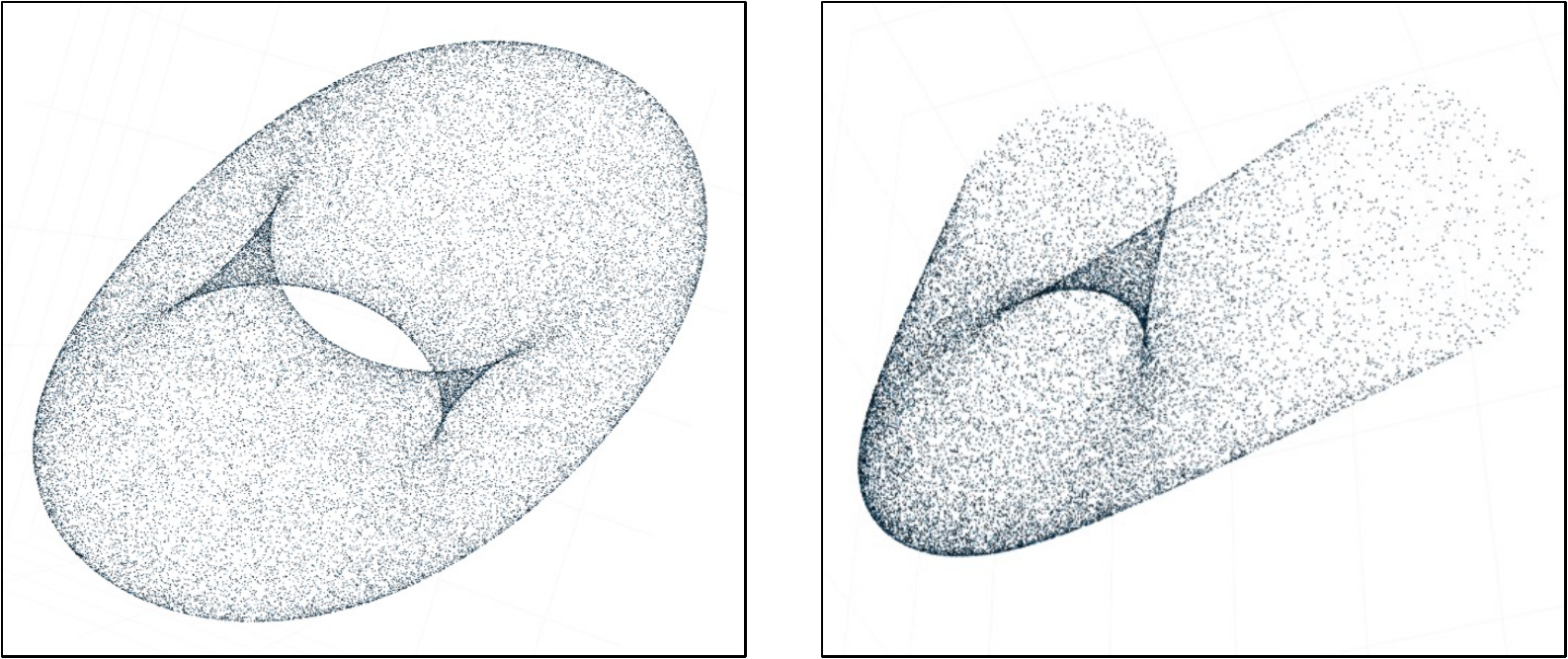}
    \caption{Scatter plot visualizations of two dimensional surfaces. Left: random points from the Bohemian dome (top left of Figure \ref{fig:4x4plots}) are plotted. Right: direct product of a circle and a hyperbola (bottom left surface of Figure \ref{fig:4x4plots}).}
    \label{fig:rand2Dscatter}
\end{figure}

\subsection{Numerical implementation of the computation}
It is well known that computing Jordan chain matrices of a given matrix is not a numerically stable procedure. Therefore, the above derivation of bases of $\sol(J)$ and $\cosol(J)$ is often unstable since it depends on a computation of the Jordan chain matrices. For a stable implementation, the staircase format \cite{kublanovskaya1966method} is preferred for revealing the eigenstructure. 

\par For a given $J$, let $J^{-T}\!J = WDW^{-1}$ and $J^{T}J^{-1} = PDP^{-1}$ be the Jordan decompositions of $\cosq(J)$ and $\cosq(J)^T$. We have $J^{-T}\!JW = WD$ and $J^TJ^{-1}P = PD$ where $D$ is the Jordan canonical form of $J^{-T}\!J$. A basis of $\cent(\cosq(J))$ discussed in Section \ref{sec:centralizer} can also be derived by the following. Let $WEP^T$ be an ansatz of a basis. Then $J^{-T}\!JWEP^T = WEP^TJ^{-T}\!J$ becomes $WDEP^T = WED^TP^T$ and $E$ has to satisfy $DE = ED^T$. When $D$ is a Jordan form, a basis of matrices $E$ satisfying $DE = ED^T$ is the set of block matrices each having a single block with its upper left corner having the backwards identity (the collection of $E_j^{s, t}$ matrices). This block nature of the $E$ matrices corresponding to the Jordan form $D$ allows us to isolate the corresponding columns of $W$ and $P$ matrices. (Lemma \ref{lem:basiscent}.)

\par Instead of the Jordan canonical form, let us consider the case where we use a numerically stable staircase form. We obtain two (unitary) matrices $W, P$ such that $J^{-T}\!JW = WT_W$, $J^TJ^{-1}P = PT_P$ where $T_W$ and $T_P$ are staircase forms. What is left to us is the similar equation $T_WE = ET_P^T$, which is a Sylvester equation. (See \cite[Chapter 16]{higham2002accuracy} for a thorough discussion on the Sylvester's equation.) To solve a Sylvester equation, one can use various software implementations most of which are based on the Bartels-Stewart algorithm \cite{bartels1972solution} and its variants. 

\par For a similar reason, the Kronecker canonical form of $J-\lambda J^T$ is also not preferred. Instead, the generalized Schur staircase format computed by the GUPTRI algorithm \cite{demmel1993generalized1,demmel1993generalized2} (or any other preferred algorithm to obtain a stable staircase format, e.g., the generalized Brunovsky canonical form \cite{morse1973structural,thorp1973singular} or the Kronecker-like form \cite{varga2020computing}) is preferred for computing the Kronecker structure of a given pencil. For a given $J$ the staircase formats of the two pencils $J-\lambda J^T$ and $J^T -\lambda J$ are $Q^H(J-\lambda J^T)W = K_W(\lambda)$ and $P^H(J-\lambda J^T)U = K_U(\lambda)$ with all $W, U, P, Q$ being unitary matrices. With two staircase forms $K_W(\lambda)$ and $K_U(\lambda)$ one needs to find a solution for $E\cdot K_W(\lambda)^T - K_U(\lambda)F=0$, and construct $(Z_1, Z_2) = (QEU^T, WFP^T)$ which satisfies $Z_1(J-\lambda J^T) - (J-\lambda J^T) Z_2 = 0$. Then, we can again collect all $Z_1^T - Z_2$ to obtain the solution set. 

%\subsection{Future works}
\section{Acknowledgements}
We thank David Vogan and Pavel Etingof for helpful discussions. We thank NSF grants OAC-1835443, OAC-2103804, SII-2029670, ECCS-2029670, PHY-2021825 for financial support.

\appendix

\section{Canonical blocks $(E, F)$}\label{sec:EFappendix}
In this section we introduce a basis of the collection of pairs $(E, F)$ (discussed in Section \ref{sec:singularJbackground}) for each canonical block of the Kronecker structure. For a single Kronecker canonical block $K(\lambda)$ we provide a basis of the pair $(E, F)$ such that $E\cdot K(\lambda)^T - K(\lambda)\cdot F=0$. For two canonical blocks (i.e., interaction) $K_1(\lambda)$ and $K_2(\lambda)$, we provide a basis of the pair $(E, F)$ such that 
\begin{equation*}
    E \cdot K_1(\lambda)^T - K_2(\lambda)\cdot F = 0 \hspace{0.3cm}\text{ or }\hspace{0.3cm}E\cdot K_2(\lambda)^T - K_1(\lambda)\cdot F = 0.
\end{equation*}

The Kronecker structure of $J-\lambda J^T$ is consisted of singular pair blocks $\mathcal{L}$ (see Definition \ref{def:auxmatrixcomplexT}) and Jordan structures. We will use a special canonical form for $\mathcal{L}$ which appears in \cite{schroder2006canonical}. For an odd $n = 2n'+1$, define an $n\times n$ matrix pencil
\begin{equation*}
    \mathcal{S}_n := \begin{bmatrix} 0 & \tilde{L}_{n'} \\L_{n'}^T  & 0 \end{bmatrix},
\end{equation*}
where $\tilde{L}_m$ is the $(m+1)\times m$ matrix pencil with ones on the diagonal and $(-\lambda)$ on the subdiagonal. The canonical block $\mathcal{S}_n$ is equivalent to the pencil $L_{n'}\oplus L_{n'}^T$. 

\par In this section the matrices $E_k^{m, n}, E_n$ in Definition \ref{def:backwardsid} are denoted by $B_k^{m, n}, B_n$ to avoid confusion. 

\subsection{$K(\lambda) = \mathcal{S}_n$}\label{sec:EFsingular1}
Let $K(\lambda) = \mathcal{S}_n$ with $n = 2n'+1$. Then a basis of the collection of $(E, F)\in(\mathbb{R}^{n\times n}, \mathbb{R}^{n\times n})$ (dimension $n+1$) is the following:
\begin{equation}\label{eq:EFsingular1}
    (E_j,F_j) = \left. \begin{cases}\Bigg(\begin{bmatrix} B_{n'+1} & \\ & 0\end{bmatrix}, \begin{bmatrix} 0 &  \\  & B_{n'}\end{bmatrix}\Bigg),\hspace{1cm} j = 1 \vspace{.2cm}\\
    \Bigg(\begin{bmatrix} 0 & 0 \\ B_{j-1}^{n', n'+1} & 0 \end{bmatrix}, \begin{bmatrix} 0& B_{j-1}^{n'+1, n'} \\ 0 & 0 \end{bmatrix}\Bigg), \hspace{0.2cm} j=2, \hdots, n \vspace{0.2cm}\\
    (F_1, E_1), \hspace{4cm} j = n+1.
    \end{cases}\right.
\end{equation}

\subsection{$K_1(\lambda) = \mathcal{S}_m, K_2(\lambda) = \mathcal{S}_n$}\label{sec:EFsingular2}
Let two canonical structures be $K_1(\lambda) = \mathcal{S}_n$ and $K_2(\lambda) = \mathcal{S}_m$, where $m>n$. Let $m', n' = \frac{m-1}{2}, \frac{n-1}{2}$. Then a basis of the collection of $(E, F)$ satisfying $E \mathcal{S}_n^T - \mathcal{S}_m F = 0$ is the following with the dimension $m$:
\begin{equation*}
    (E_j,F_j) = \left. \begin{cases}\Bigg(\begin{bmatrix} B_{n'+j}^{m'+1, n'+1} & \\ & 0\end{bmatrix}, \begin{bmatrix} 0 & \\ & B_{n'+j-1}^{m', n'}\end{bmatrix}\Bigg),\hspace{1cm} j = 1, \hdots, \frac{m-n}{2}+1 \vspace{0.2cm}\\
    \Bigg(\begin{bmatrix} 0 & 0 \\ B_{j-(m'-n'+1)}^{m', n'+1} & 0 \end{bmatrix}, \begin{bmatrix} 0& B_{j-(m'-n'+1)}^{m'+1, n'} \\ 0 & 0 \end{bmatrix}\Bigg), \hspace{0.2cm} j=\frac{m-n}{2}+2, \hdots, m.
    \end{cases}\right.
\end{equation*}
For $E \mathcal{S}_m^T - \mathcal{S}_n F = 0$ we have a dimension $m$ basis set for $(E, F)$:
\begin{equation*}
    (E_j, F_j) = (F_{j-m}^T, E_{j-m}^T),\hspace{1cm}j = m+1, \dots, 2m.
\end{equation*}
The total dimension is $2m$. For $m=n$ we add $(F_1, E_1)$ as $(E_{m+1}, F_{m+1})$ to obtain the total dimension $2m+2$. 

\subsection{$\mathcal{S}_n$ and a Jordan block}\label{sec:EFsingularandJordan}
Let two canonical structures be $K_1(\lambda) = \mathcal{S}_n$ and $K_2(\lambda) = J^\alpha_m - \lambda I_m$ where the Jordan block has the eigenvalue $\alpha$. Again let $n = 2n'+1$. First define a matrix $G_m^\alpha$,
\begin{equation*}
    [G_m^\alpha]_{jk} = \binom{j}{k+j-m}\alpha^{k+j-m}.
\end{equation*}
For example, if $m = 3$, then 
\begin{equation*}
    G_3^\alpha = \begin{bmatrix}0 & 0 & 1\\ 0 & 1 & \alpha \\ 1 & 2\alpha & \alpha^2\end{bmatrix}.
\end{equation*}
Then a basis of the collection of $(E, F)\in(\mathbb{R}^{n\times m}, \mathbb{R}^{n\times m})$ (dimension $m$) can be described as the following. For a fixed $j \leq m$, $E_j$ is the matrix with its $(j-k+1)^{th}$ column having its bottom $n'$ entries equal to the $(n'-k+1)^{th}$ column of $G_{n'}^\alpha$, for $k = 1, \dots, n'$. (Ignore negative indexed columns and just discard them.) Also, $F_j$ is the matrix with its $(j-k+1)^{th}$ column having its top $n'+1$ entries equal to the $(n'-k+2)^{th}$ column of $G_{n'+1}^\alpha$ for $k = 1, \dots, n'+1$. For example when $m=7, n=5$, we have 
\begin{equation*}
    E_2, F_2 = \begin{bmatrix} 0 & 0 & 0 & 0 & 0 & 0 & 0\\
    0 & 0 & 0 & 0 & 0 & 0 & 0 \\
     0 & 0 & 0 & 0 & 0 & 0 & 0\\
      0 & 1 & 0 & 0 & 0 & 0 & 0\\
       1 & \alpha & 0 & 0 & 0 & 0 & 0\\ \end{bmatrix}, 
       \begin{bmatrix} 0 & 1 & 0 & 0 & 0 & 0 & 0 \\
     1 & \alpha & 0 & 0 & 0 & 0 & 0\\
      2\alpha & \alpha^2 & 0 & 0 & 0 & 0 & 0\\
      0 & 0 & 0 & 0 & 0 & 0 & 0\\
       0 & 0 & 0 & 0 & 0 & 0 & 0\\ \end{bmatrix}.
\end{equation*}
Also for $E K_2(\lambda)^T - K_1(\lambda) F = 0$ we have 
\begin{equation*}
     (E_j, F_j) = (F_{j-m}^T, E_{j-m}^T),\hspace{1cm}j = m+1, \dots, 2m.
\end{equation*}
The total dimension is $2m$.

\section{Proofs of Propositions \ref{thm:choiceofWU} and \ref{thm:XstlambdaXtslambdai}}\label{sec:proofs}

\subsection{Proof of Proposition \ref{thm:choiceofWU}}
\begin{proof}
We prove the result for $X_T$ with two different $W$'s. The other cases are similar. Let $W$ and $W'$ be two $n\times r$ Jordan chain matrices of $C$. Realizing Jordan chain as a chain of basis elements of the nullspaces $\text{null}(C-\lambda I)^j$ there exists an invertible upper triangular $T$ such that $WT = W'$. From $CW' = W'J_r^\lambda$ we have $CWT = WTJ_r^\lambda = WJ_r^\lambda T$. Since $W$ is of full rank we deduce $J_r^\lambda T = TJ_r^\lambda$, which makes $T$ a Toeplitz matrix. ($\cent(J_r^\lambda)$ is the set of upper triangular Toeplitz matrices.) The $r\times r$ upper triangular Toeplitz matrix $T$ can be expressed as $T = \sum t_j E_j^{r, r}E_r^{r, r}$.
\par Now let us consider two matrices $X_T(k, J, W, U)$ and $X_T(k, J, W', U)$ where $U$ is an $n\times m$ Jordan chain matrix of $C^{-1}$. We have
\begin{align*}
    X_T(k, J, W', U) &= WTE_k^{r, m}U^TJ^T - UE_k^{m, r}T^TW^TJ \\
    &= \sum_{j=1}^k t_{r+j-k} X_T(j, J, W, U),
\end{align*}
which proves the sets $\bigcup_{j=1}^k\big\{X_T(j, J, W', U)\big\}$ and $\bigcup_{j=1}^k\big\{X_T(j, J, W, U)\big\}$ span the same vector space. 
\end{proof}

\subsection{Proof of Proposition \ref{thm:XstlambdaXtslambdai}}
\begin{proof}
For a fixed $j$, let $\big(J_{r_j}^{\lambda}\big)^{-1} = K_j J_{r_j}^{1/\lambda} K_j^{-1}$ be the Jordan decomposition of $\big(J_{r_j}^{\lambda}\big)^{-1}$. From $CW_\lambda^{(j)} = W_\lambda^{(j)}J_{r_j}^\lambda$ we obtain $C^{-1}\big(W_\lambda^{(j)}K_j\big) = \big(W_\lambda^{(j)}K_j\big)J_{r_j}^{1/\lambda}$ which makes $\big(W_\lambda^{(j)}K_j\big)$ as an eligible choice of $U_{\frac{1}{\lambda}}^{(j)}$. Similarly $\big(U_\lambda^{(j)}K_j\big)$ is an eligible choice of $W_{\frac{1}{\lambda}}^{(j)}$. Finally for $1\leq s, t \leq m$ we have
\begin{equation*}
    X_T(k,J,W_\frac{1}{\lambda}^{(t)}, U_\frac{1}{\lambda}^{(s)}) = X_T(k, J, \tilde{U}_\lambda^{(t)}K_t, \tilde{W}_\lambda^{(s)}K_s),
\end{equation*}
for some choice of Jordan chain matrices $\tilde{U}_\lambda^{(t)}, \tilde{W}_\lambda^{(s)}$. By Proposition \ref{thm:choiceofWU} and \eqref{eq:XTsubstitute}, we need show $K_sE_k^{r_s, r_t}K_t^T = \sum_{l=1}^k c_l E_k^{r_s, r_t}$ for some scalars $c_l$. This turns out to be true since matrix $H := K_sE_k^{r_s, r_t}K_t^TE_k^{r_t, r_s}$ is an upper triangular Toeplitz matrix, by the definition of $K_s, K_t$ and the fact that $H$ commutes with $J_{r_s}^\lambda$. 
\par A similar technique also proves the results for $Y_T, X_H, Y_H, X_\mathbb{R}, Y_\mathbb{R}$. 
\end{proof}

\bibliographystyle{siam}
\bibliography{bibliography.bib}

\begin{thebibliography}{10}

\bibitem{arnold1971matrices}
{\sc V.~I. Arnold}, {\em On matrices depending on parameters}, Russian
  Mathematical Surveys, 26 (1971), pp.~29--43.

\bibitem{bartels1972solution}
{\sc R.~H. Bartels and G.~W. Stewart}, {\em Solution of the matrix equation
  {$AX+ XB= C$}}, Communications of the ACM, 15 (1972), pp.~820--826.

\bibitem{benner1998numerically}
{\sc P.~Benner, V.~Mehrmann, and H.~Xu}, {\em A numerically stable, structure
  preserving method for computing the eigenvalues of real {H}amiltonian or
  symplectic pencils}, Numerische Mathematik, 78 (1998), pp.~329--358.

\bibitem{bezanson2017julia}
{\sc J.~Bezanson, A.~Edelman, S.~Karpinski, and V.~B. Shah}, {\em Julia: A
  fresh approach to numerical computing}, SIAM review, 59 (2017), pp.~65--98.

\bibitem{bunse1992chart}
{\sc A.~Bunse-Gerstner, R.~Byers, and V.~Mehrmann}, {\em A chart of numerical
  methods for structured eigenvalue problems}, SIAM Journal on Matrix Analysis
  and Applications, 13 (1992), pp.~419--453.

\bibitem{chan2013matrix}
{\sc A.~Z. Chan, L.~A.~G. German, S.~R. Garcia, and A.~L. Shoemaker}, {\em On
  the matrix equation {$XA+ AX^T= 0$}, {II}: Type {$0-I$} interactions}, Linear
  Algebra and its Applications, 439 (2013), pp.~3934--3944.

\bibitem{chu1982exponential}
{\sc H.~Chu}, {\em On the exponential map of classical {L}ie groups and linear
  differential systems with periodic coefficients}, Journal of Mathematical
  Analysis and Applications, 85 (1982), pp.~566--583.

\bibitem{DanischKrumbiegel2021}
{\sc S.~Danisch and J.~Krumbiegel}, {\em Makie.jl: Flexible high-performance
  data visualization for {J}ulia}, Journal of Open Source Software, 6 (2021),
  p.~3349.

\bibitem{de2016canonical}
{\sc F.~De~Ter{\'a}n}, {\em Canonical forms for congruence of matrices and
  {$T$}-palindromic matrix pencils: A tribute to {H. W. Turnbull} and {A. C.
  Aitken}}, SeMA Journal, 73 (2016), pp.~7--16.

\bibitem{de2011conjtrans}
{\sc F.~De~Ter{\'a}n and F.~M. Dopico}, {\em The equation {$XA+ AX^*= 0$} and
  the dimension of * congruence orbits}, The Electronic Journal of Linear
  Algebra, 22 (2011), pp.~448--465.

\bibitem{de2011trans}
\leavevmode\vrule height 2pt depth -1.6pt width 23pt, {\em The solution of the
  equation {$XA+ AX^T= 0$} and its application to the theory of orbits}, Linear
  Algebra and its Applications, 434 (2011), pp.~44--67.

\bibitem{de2013solution}
{\sc F.~De~Ter{\'a}n, F.~M. Dopico, N.~Guillery, D.~Montealegre, and N.~Reyes},
  {\em The solution of the equation {$AX+ X^\star B$= 0}}, Linear Algebra and
  its Applications, 438 (2013), pp.~2817--2860.

\bibitem{demmel1995dimension}
{\sc J.~W. Demmel and A.~Edelman}, {\em The dimension of matrices (matrix
  pencils) with given {J}ordan ({K}ronecker) canonical forms}, Linear Algebra
  and its Applications, 230 (1995), pp.~61--87.

\bibitem{demmel1993generalized1}
{\sc J.~W. Demmel and B.~K{\aa}gstr{\"o}m}, {\em The generalized {S}chur
  decomposition of an arbitrary pencil {$A-\lambda B$}: Robust software with
  error bounds and applications. {P}art {I}: theory and algorithms}, ACM
  Transactions on Mathematical Software (TOMS), 19 (1993), pp.~160--174.

\bibitem{demmel1993generalized2}
\leavevmode\vrule height 2pt depth -1.6pt width 23pt, {\em The generalized
  {S}chur decomposition of an arbitrary pencil {$A-\lambda B$}: Robust software
  with error bounds and applications. {P}art {II}: software and applications},
  ACM Transactions on Mathematical Software (TOMS), 19 (1993), pp.~175--201.

\bibitem{dmytryshyn2015change}
{\sc A.~Dmytryshyn, V.~Futorny, B.~K{\aa}gstr{\"o}m, L.~Klimenko, and V.~V.
  Sergeichuk}, {\em Change of the congruence canonical form of 2-by-2 and
  3-by-3 matrices under perturbations and bundles of matrices under
  congruence}, Linear Algebra and its Applications, 469 (2015), pp.~305--334.

\bibitem{dmytryshyn2013codimension}
{\sc A.~Dmytryshyn, S.~Johansson, and B.~K{\aa}gstr{\"o}m}, {\em Codimension
  computations of congruence orbits of matrices, symmetric and skew-symmetric
  matrix pencils using Matlab}, Ume{\aa} Universitet, 2013.

\bibitem{fassbender2007symplectic}
{\sc H.~Fa{\ss}bender}, {\em Symplectic methods for the symplectic
  eigenproblem}, Springer Science \& Business Media, 2007.

\bibitem{futorny2014change}
{\sc V.~Futorny, L.~Klimenko, and V.~Sergeichuk}, {\em Change of
  the*-congruence canonical form of 2-by-2 matrices under perturbations}, The
  Electronic Journal of Linear Algebra, 27 (2014), pp.~146--154.

\bibitem{gantmacher1964theory}
{\sc F.~R. Gantmacher}, {\em The Theory of Matrices}, vol.~1, Chelsea, 1964.

\bibitem{garcia2013matrix}
{\sc S.~R. Garcia and A.~L. Shoemaker}, {\em On the matrix equation {$XA+ AX^T=
  0$}}, Linear Algebra and its Applications, 438 (2013), pp.~2740--2746.

\bibitem{gohberg2006invariant}
{\sc I.~Gohberg, P.~Lancaster, and L.~Rodman}, {\em Invariant Subspaces of
  Matrices with Applications}, SIAM, 2006.

\bibitem{Helgason1978}
{\sc S.~Helgason}, {\em Differential Geometry, {L}ie Groups, and Symmetric
  Spaces}, Academic Press, 1979.

\bibitem{higham2002accuracy}
{\sc N.~J. Higham}, {\em Accuracy and Stability of Numerical Algorithms}, SIAM,
  2002.

\bibitem{horn2012matrix}
{\sc R.~A. Horn and C.~R. Johnson}, {\em Matrix Analysis}, Cambridge university
  press, 2012.

\bibitem{horn2007canonical}
{\sc R.~A. Horn and V.~V. Sergeichuk}, {\em Canonical matrices of bilinear and
  sesquilinear forms}, Linear Algebra and its Applications, 428 (2008),
  pp.~193--223.

\bibitem{jacobson2009basic}
{\sc N.~Jacobson}, {\em Basic Algebra}, vol.~1, Courier Corporation, 2009.

\bibitem{sungwoo2021}
{\sc S.~Jeong}, {\em Linear Algebra, Random Matrices and Lie Theory}, PhD
  thesis, Massachusetts Institute of Technology, 2021.

\bibitem{kagstrom1982matrix}
{\sc B.~Kagstr{\"o}m and A.~Ruhe}, {\em Matrix Pencils}, Springer-Verlag, New
  York, 1982.

\bibitem{kublanovskaya1966method}
{\sc V.~N. Kublanovskaya}, {\em On a method of solving the complete eigenvalue
  problem for a degenerate matrix}, USSR Computational Mathematics and
  Mathematical Physics, 6 (1966), pp.~1--14.

\bibitem{lai1977surjectivity}
{\sc H.-L. Lai}, {\em Surjectivity of exponential map on semisimple {L}ie
  groups}, Journal of the Mathematical Society of Japan, 29 (1977),
  pp.~303--325.

\bibitem{lang2002algebra}
{\sc S.~Lang}, {\em Algebra}, vol.~211 of Graduate Texts in Mathematics,
  Springer, 2002.

\bibitem{macduffee1933}
{\sc C.~C. MacDuffee}, {\em The Theory of Matrices}, Springer-Verlag, Berlin,
  1933.

\bibitem{mackey2003structured}
{\sc D.~S. Mackey, N.~Mackey, and F.~Tisseur}, {\em Structured tools for
  structured matrices}, The Electronic Journal of Linear Algebra, 10 (2003),
  pp.~106--145.

\bibitem{mackey2005structured}
\leavevmode\vrule height 2pt depth -1.6pt width 23pt, {\em Structured
  factorizations in scalar product spaces}, SIAM Journal on Matrix Analysis and
  Applications, 27 (2005), pp.~821--850.

\bibitem{morse1973structural}
{\sc A.~S. Morse}, {\em Structural invariants of linear multivariable systems},
  SIAM Journal on Control, 11 (1973), pp.~446--465.

\bibitem{djokovic1997surjectivity}
{\sc D.~{\v{Z}}. {\DJ}okovi{\'c} and K.~H. Hofmann}, {\em The surjectivity
  question for the exponential function of real {L}ie groups: A status
  report.}, Journal of Lie Theory, 7 (1997), pp.~171--199.

\bibitem{riehm1974equivalence}
{\sc C.~Riehm}, {\em The equivalence of bilinear forms}, Journal of Algebra, 31
  (1974), pp.~45--66.

\bibitem{rossmann2002lie}
{\sc W.~Rossmann}, {\em Lie Groups: An Introduction through Linear Groups},
  vol.~5 of Oxford Graduate Texts in Mathematics, Oxford University Press,
  2002.

\bibitem{schroder2006canonical}
{\sc C.~Schr{\"o}der}, {\em A canonical form for palindromic pencils and
  palindromic factorizations}, Preprint 316, TU Berlin, Matheon,  (2006).

\bibitem{schroder2008thesis}
\leavevmode\vrule height 2pt depth -1.6pt width 23pt, {\em Palindromic and Even
  Eigenvalue Problems-Analysis and Numerical Methods}, PhD thesis, TU Berlin,
  2008.

\bibitem{sibuya1960note}
{\sc Y.~Sibuya}, {\em Note on real matrices and linear dynamical systems with
  periodic coefficients}, Journal of Mathematical Analysis and Applications, 1
  (1960), pp.~363--372.

\bibitem{szechtman2005structure}
{\sc F.~Szechtman}, {\em Structure of the group preserving a bilinear form},
  The Electronic Journal of Linear Algebra, 13 (2005), pp.~197--239.

\bibitem{taussky1979some}
{\sc O.~Taussky}, {\em Some remarks concerning matrices of the form {$A-A',
  A^{-1}A'$}}, Zeitschrift f{\"u}r angewandte Mathematik und Physik (ZAMP), 30
  (1979), pp.~370--373.

\bibitem{thorp1973singular}
{\sc J.~S. Thorp}, {\em The singular pencil of a linear dynamical system},
  International Journal of Control, 18 (1973), pp.~577--596.

\bibitem{turnbull1932introduction}
{\sc H.~W. Turnbull and A.~C. Aitken}, {\em An Introduction to the Theory of
  Canonical Matrices}, Blackie \& Son Ltd., 1932.

\bibitem{van1979computation}
{\sc P.~Van~Dooren}, {\em The computation of {K}ronecker's canonical form of a
  singular pencil}, Linear Algebra and Its Applications, 27 (1979),
  pp.~103--140.

\bibitem{van1983reducing}
\leavevmode\vrule height 2pt depth -1.6pt width 23pt, {\em Reducing subspaces:
  Definitions, properties and algorithms}, in Matrix Pencils, Springer, 1983,
  pp.~58--73.

\bibitem{varga2020computing}
{\sc A.~Varga}, {\em On computing the kronecker structure of polynomial and
  rational matrices using {J}ulia}, arXiv preprint arXiv:2006.06825,  (2020).

\bibitem{weyl1946classical}
{\sc H.~Weyl}, {\em The Classical Groups: Their Invariants and
  Representations}, vol.~45, Princeton University Press, 1946.

\end{thebibliography}

\end{document}